\documentclass[UKenglish, letterpaper]{amsart} 

\usepackage{amscd}
\usepackage{amssymb}
\usepackage{babel}
\usepackage{cancel}
\usepackage[T1]{fontenc}
\usepackage[utf8]{inputenc}
\usepackage{hyperref}
\usepackage[hyperpageref]{backref}
\usepackage{varioref}
\usepackage[arrow,curve,matrix]{xy}

\usepackage{mathpazo}
\usepackage{mathrsfs}

\usepackage{graphicx}
\usepackage{color}

\def\clap#1{\hbox to 0pt{\hss#1\hss}}

\DeclareFontFamily{OMS}{rsfs}{\skewchar\font'60}
\DeclareFontShape{OMS}{rsfs}{m}{n}{<-5>rsfs5 <5-7>rsfs7 <7->rsfs10 }{}
\DeclareSymbolFont{rsfs}{OMS}{rsfs}{m}{n}
\DeclareSymbolFontAlphabet{\scr}{rsfs}

 \newcommand{\C}{\mathbb{C}}
\renewcommand{\O}{\sO}

\newcommand{\sA}{\scr{A}}

\newcommand{\sE}{\scr{E}}
\newcommand{\sF}{\scr{F}}
\newcommand{\sG}{\scr{G}}

\newcommand{\sI}{\scr{I}}

\newcommand{\sL}{\scr{L}}

\newcommand{\sO}{\scr{O}}

\newcommand{\sQ}{\scr{Q}}

\newcommand{\sS}{\scr{S}}
\newcommand{\sT}{\scr{T}}

\newcommand{\bC}{\mathbb{C}}

\newcommand{\bN}{\mathbb{N}}

\newcommand{\bQ}{\mathbb{Q}}

\DeclareMathOperator{\Alb}{Alb}
\DeclareMathOperator{\complete}{complete}
\DeclareMathOperator{\restX}{rest}

\DeclareMathOperator{\codim}{codim}

\DeclareMathOperator{\Hom}{Hom}

\DeclareMathOperator{\Id}{Id}
\DeclareMathOperator{\Image}{Image}

\DeclareMathOperator{\Pic}{Pic}
\DeclareMathOperator{\rank}{rank}

\DeclareMathOperator{\reg}{reg}

\DeclareMathOperator{\Sym}{Sym}
\DeclareMathOperator{\supp}{supp}

\DeclareMathOperator{\Aut}{Aut}

\newcommand{\into}{\hookrightarrow}

\newcommand{\wtilde}{\widetilde}
\newcommand{\what}{\widehat}

\newcounter{thisthm}

\newcommand{\iref}[1]{(\thesection.\the\value{thisthm}.\the\value{#1})}

\theoremstyle{plain}    
\newtheorem{thm}{Theorem}[section]

\numberwithin{equation}{thm}
\numberwithin{figure}{section}
\theoremstyle{plain}    
\newtheorem{cor}[thm]{Corollary}
\newtheorem{lem}[thm]{Lemma}

\newtheorem{problem}[thm]{Problem}

\theoremstyle{plain}    
\newtheorem{prop}[thm]{Proposition}
\newtheorem{proclaim-special}[thm]{\specialthmname}

\theoremstyle{remark}
\newtheorem{setup}[thm]{Setup} 
\newtheorem{ex}[thm]{Example}
\newtheorem{defn}[thm]{Definition}
\newtheorem{rem}[thm]{Remark}
\newtheorem{obs}[thm]{Observation}

\newtheorem{claim}[thm]{Claim} 
\newtheorem*{claim*}{Claim} 
\newtheorem{notation}[thm]{Notation}

\newtheorem{conclusion}[thm]{Conclusion}

\newtheoremstyle{bozont-remark}{3pt}{3pt}%
     {}
     {}
     {\it}
     {.}
     {.5em}
     {\thmname{#1}\thmnumber{ #2}: \thmnote{\sc #3}}
\theoremstyle{bozont-remark}

\def\factor#1.#2.{\left. \raise 2pt\hbox{$#1$} \right/\hskip -2pt\raise
  -2pt\hbox{$#2$}} 
\newlength{\swidth}
\newenvironment{enumerate-p}{
  \begin{enumerate}}
  {\setcounter{equation}{\value{enumi}}\end{enumerate}}

\definecolor{tomato}{RGB}{180,62,39}
\definecolor{forrest}{RGB}{81,133,49}
\definecolor{lighttomato}{RGB}{253,65,65}
\definecolor{lightforrest}{RGB}{145,237,87}
\definecolor{mygreen}{RGB}{40,104,69}
\definecolor{mygreen2}{RGB}{3,149,39}
\definecolor{darkolivegreen}{RGB}{102,118,75}
\definecolor{cranegreen}{RGB}{102,118,75}
\definecolor{mydarkblue}{RGB}{10,92,153}
\definecolor{myblue}{RGB}{57,222,186}
\definecolor{pinkish}{RGB}{213,83,222}
\definecolor{colD}{RGB}{213,83,222}
\definecolor{defb}{RGB}{213,83,222}
\definecolor{goldenrod}{RGB}{225,115,69}
\definecolor{mauve}{RGB}{224, 176, 255}
\definecolor{fuchsia}{RGB}{255, 0, 255}
\definecolor{lavender}{RGB}{230, 230, 250}
\definecolor{gold}{RGB}{255, 215, 0}
\definecolor{orange}{RGB}{255, 127, 0}
\definecolor{maroon}{RGB}{123, 17, 19}
\definecolor{brightmaroon}{RGB}{195, 33, 72}
\definecolor{richmaroon}{RGB}{176, 48, 96}
\definecolor{green}{RGB}{3,149,39}

\title{Singular spaces with trivial canonical class}
\dedicatory{Dedicated to Professor Shigefumi Mori on the occasion
     of his $60^{\rm th}$ birthday}
\author{Daniel Greb}
\author{Stefan Kebekus}
\author{Thomas Peternell}

\thanks{All three authors were supported in part by the DFG-Forschergruppe 790
  ``Classification of Algebraic Surfaces and Compact Complex Manifolds''. During
  the preparation of parts of this paper, the first named author enjoyed the
  hospitality of the Mathematics Department at Princeton University and the
  Department of Pure Mathematics \& Mathematical Statistics at University of
  Cambridge. He gratefully acknowledges the support of the
  Baden--Württemberg--Stiftung through the ``Eliteprogramm für
  Postdoktorandinnen und Postdoktoranden''.}

\address{Daniel Greb, Mathematisches Institut, Albert-Ludwigs-Universität
  Freiburg, Eckerstraße 1, 79104 Freiburg im Breisgau, Germany}
\email{\href{mailto:daniel.greb@math.uni-freiburg.de}{daniel.greb@math.uni-freiburg.de}}
\urladdr{\href{http://home.mathematik.uni-freiburg.de/dgreb}{http://home.mathematik.uni-freiburg.de/dgreb}}

\address{Stefan Kebekus, Mathematisches Institut, Albert-Ludwigs-Universität
  Freiburg, Eckerstraße 1, 79104 Freiburg im Breisgau, Germany}
\email{\href{mailto:stefan.kebekus@math.uni-freiburg.de}{stefan.kebekus@math.uni-freiburg.de}}
\urladdr{\href{http://home.mathematik.uni-freiburg.de/kebekus}{http://home.mathematik.uni-freiburg.de/kebekus}}

\address{Thomas Peternell, Institut für Mathematik, Universität Bayreuth,
  95440~Bayreuth, Germany}
\email{\href{mailto:thomas.peternell@uni-bayreuth.de}{thomas.peternell@uni-bayreuth.de}}
\urladdr{\href{http://www.staff.uni-bayreuth.de/~btm109/peternell/pet.html}{http://www.staff.uni-bayreuth.de/$\sim$btm109/peternell/pet.html}}

\date{\today}



\hypersetup{
  pdfauthor={Daniel Greb, Stefan Kebekus, and Thomas Peternell},
  pdftitle={Singular spaces with trivial canonical class},
  pdfkeywords={Beauville-Bogomolov Decomposition Theorem, Varieties of Kodaira Dimension Zero, Calabi-Yau, holomorphic-symplectic Varieties, Minimal Model Theory},
  pdfstartview={Fit},
  pdfpagelayout={TwoColumnRight},
  pdfpagemode={UseOutlines},
  bookmarks,
  colorlinks}

\begin{document}

\begin{abstract}
  The classical Beauville-Bogomolov Decomposition Theorem asserts that any
  compact Kähler manifold with numerically trivial canonical bundle admits an
  étale cover that decomposes into a product of a torus, and irreducible,
  simply-connected Calabi-Yau-- and holomorphic-symplectic manifolds. The
  decomposition of the simply-connected part corresponds to a decomposition of
  the tangent bundle into a direct sum whose summands are integrable and stable
  with respect to any polarisation.

  Building on recent extension theorems for differential forms on singular
  spaces, we prove an analogous decomposition theorem for the tangent sheaf of
  projective varieties with canonical singularities and numerically trivial
  canonical class.

  In view of recent progress in minimal model theory, this result can be seen as
  a first step towards a structure theory of manifolds with Kodaira dimension
  zero. Based on our main result, we argue that the natural building blocks for
  any structure theory are two classes of canonical varieties, which generalise
  the notions of irreducible Calabi-Yau-- and irreducible holomorphic-symplectic
  manifolds, respectively.
\end{abstract}

\maketitle

\vspace{0.3cm}
\tableofcontents
\newpage
\section{Introduction}

\subsection{Introduction and main result}
\label{ssec:IMR}

The minimal model program aims to reduce the birational study of projective
manifolds with Kodaira dimension zero to the study of associated \emph{minimal
  models}, that is, normal varieties $X$ with terminal singularities whose
canonical divisor is numerically trivial. The ideal case, where the minimal
variety $X$ is smooth, is described in the fundamental Decomposition Theorem of
Beauville and Bogomolov.

\begin{thm}[\protect{Beauville-Bogomolov Decomposition, \cite{Bea83} and references there}]\label{bb}%
  Let $X$ be a compact Kähler manifold whose canonical divisor is numerically
  trivial. Then there exists a finite étale cover, $X' \to X$ such that $X'$
  decomposes as a product
  \begin{equation}\label{eq:BBdc}
    X' = T \times \prod_{\nu} X_{\nu}    
  \end{equation}
  where $T$ is a compact complex torus, and where the $X_{\nu} $ are irreducible
  and simply-connected Calabi-Yau-- or holomorphic-symplectic manifolds.
\end{thm}

\begin{rem}
  The decomposition~\eqref{eq:BBdc} induces a decomposition of the tangent
  bundle $T_{X'}$ into a direct sum whose summands have vanishing Chern class,
  and are integrable in the sense of Frobenius' theorem. Those summands that
  correspond to the $X_{\nu}$ are slope-stable with respect to any ample
  polarisation.
\end{rem}

In view of recent progress in minimal model theory, an analogue of
Theorem~\ref{bb} for minimal models would clearly be a substantial step towards
a complete structure theory for varieties of Kodaira dimension zero. However,
since the only known proof of Theorem~\ref{bb} heavily uses Kähler-Einstein
metrics and the solution of the Calabi conjecture, a full generalisation of
Beauville-Bogomolov Decomposition Theorem~\ref{bb} to the singular setting is
difficult.

\subsubsection*{Main result}

The main result of our paper is the following Decomposition Theorem for the
tangent sheaf of minimal varieties with vanishing first Chern class. Presenting
the tangent sheaf as a direct sum of integrable subsheaves which are stable with
respect to any polarisation, Theorem~\ref{decoII} can be seen as an
infinitesimal analogue of the Beauville-Bogomolov Decomposition~\ref{bb} in the
singular setting.

\begin{thm}[Decomposition of the tangent sheaf]\label{decoII}%
  Let $X$ be a normal projective variety with at worst canonical singularities,
  defined over the complex numbers. Assume that the canonical divisor of $X$ is
  numerically trivial: $K_X \equiv 0$. Then there exists an Abelian variety $A$
  as well as a projective variety $\wtilde X$ with at worst canonical
  singularities, a finite cover $f: A \times \wtilde X \to X$, étale in
  codimension one, and a decomposition
  $$
  \sT_{\wtilde X} \cong \bigoplus \sE_i
  $$
  such that the following holds.
  \begin{enumerate-p}
  \item\label{il:A} The $\sE_i$ are integrable saturated subsheaves of
    $\sT_{\wtilde X}$, with trivial determinants.
  \end{enumerate-p}
  Further, if $g: \what X \to \wtilde X$ is any finite cover, étale in
  codimension one, then the following properties hold in addition.
  \begin{enumerate-p}
    \setcounter{enumi}{\value{equation}}
  \item\label{il:B} The sheaves $(g^* \sE_i)^{**}$ are slope-stable with respect
    to any ample polarisation on $\what X$.
  \item\label{il:C} The irregularity of $\what X$ is zero, $h^1\bigl( \what X,\,
    \sO_{\what X} \bigr) = 0$.
  \end{enumerate-p}
\end{thm}
The decomposition found in Theorem~\ref{decoII} satisfies an additional
uniqueness property. For a precise statement, see Remark~\vref{rem:uniqueness}.
Taking $g$ to be the identity, we see that the irregularity of $\wtilde X$ is
zero, and that the summands $\sE_i$ are stable with respect to any polarisation.

\subsubsection*{Other results}

In the course of the proof, we show the following two additional results,
pertaining to stability of the tangent bundle, and to wedge products of
differentials forms that are defined on the smooth locus of a minimal model. We
feel that these results might be of independent interest.

\begin{prop}[Stability of $\sT_X$ does not depend on polarisation, Proposition~\ref{prop:indepSS}]
  Let $X$ be a normal projective variety having at worst canonical
  singularities. Assume that $K_X$ is numerically equivalent to zero.  If the
  tangent sheaf $\sT_X$ is slope-stable with respect to one polarisation, then
  it is also stable with respect to any other polarisation. \qed
\end{prop}

\begin{prop}[Non-degeneracy of the wedge product, Proposition~\ref{prop:forms-1}]
  Let $X$ be a normal $n$-dimensional projective variety $X$ having at worst
  canonical singularities. Denote the smooth locus of $X$ by $X_{\reg}$. Suppose
  that the canonical divisor is trivial. If $0 \leq p \leq n$ is any number,
  then the natural pairing given by the wedge product on $X_{\reg}$,
  $$
  \bigwedge : H^0\bigl( X_{\reg},\, \Omega^p_{X_{\reg}} \bigr) \times H^0\bigl(
  X_{\reg},\, \Omega^{n-p}_{X_{\reg}} \bigr) \longrightarrow H^0\bigl(
  X_{\reg},\, \omega_{X_{\reg}} \bigr) \cong H^0\bigl( X,\, \omega_X \bigr)
  \cong \bC,
  $$
  is non-degenerate. \qed
\end{prop}

\subsubsection*{Singular analogues of Calabi-Yau and irreducible symplectic manifolds}

Based on the Decomposition Theorem~\ref{decoII}, we will argue in
Section~\ref{subsect:fromtangentbundletoX} that the natural building blocks for
any structure theory of projective manifolds with Kodaira dimension zero are
canonical varieties with \emph{strongly stable} tangent sheaf. Strong stability,
introduced in Definition~\vref{def:strongStab}, is a formalisation of condition
(\ref{decoII}.\ref{il:B}) that appears in the Decomposition
Theorem~\ref{decoII}.

In dimension no more than five, we show that canonical varieties with strongly
stable tangent sheaf fall into two classes, which naturally generalise the
notions of irreducible Calabi-Yau-- and irreducible holomorphic-symplectic
manifolds, respectively. There is ample evidence to suggest that this dichotomy
should hold in arbitrary dimension.

\subsubsection*{Outline of the paper}

The proof of Theorem~\ref{decoII} relies on recent extension results for
differential forms on singular spaces, which we recall in
Section~\ref{sec:reflDiff} below.  There are three additional preparatory
sections, Sections \ref{sec:known}--\ref{sect:stability}, where we recall
structure results for varieties with trivial canonical bundle, and discuss
stability properties of the tangent sheaf on varieties with numerically trivial
canonical divisor. Some of the material in these sections is new.

Using the extension result together with Hodge-theoretic arguments, we will show
in Section~\ref{sec:kawamata} that the wedge-product induces perfect pairings of
reflexive differential forms. This will later on be used to split the inclusion
of certain subsheaves of the tangent sheaf. The results obtained there
generalise ideas of Bogomolov \cite{Bogomolov74}, but are new in the singular
setting, to the best of our knowledge. With these preparations in place,
Theorem~\ref{decoII} is then shown in Section~\ref{sect:deco}.

Based on our main results, the concluding
Section~\ref{subsect:fromtangentbundletoX} discusses possible approaches towards
a more complete structure theory of singular varieties with trivial canonical
bundle, and proves first results in this direction.

\subsubsection*{Global Convention}

Throughout the paper we work over the complex number field. In the discussion of
sheaves, the word ``stable'' will always mean ``slope-stable with respect to a
given polarisation''. Ditto for semistability.

\subsection{Acknowledgements}

The authors would like to thank Tommaso de Fernex, Alex Küronya and Keiji Oguiso
for a number of discussions, and for answering our questions. The authors thank
the referee for reading the manuscript with great care. The first named author
wants to thank János Kollár and Burt Totaro for their invitations and for useful
comments concerning the topics discussed in this paper.

\section{Reflexive differentials on normal spaces}
\label{sec:reflDiff}

\subsection{Differentials, reflexive tensor operations}

Given a normal variety $X$, we denote the sheaf of Kähler differentials by
$\Omega^1_X$. The tangent sheaf will be denoted by $\sT_X =
(\Omega^1_X)^*$. Reflexive differentials, as defined below, will play an
important role in the discussion.

\begin{defn}[\protect{Reflexive differential forms, cf.~\cite[Sect.~2.E]{GKKP10}}]\label{def:reflDiff}
  Let $X$ be a normal variety, let $X_{\reg}$ be the smooth locus of $X$ and
  $\imath: X_{\reg} \hookrightarrow X$ its open embedding into $X$. If $0 \leq p
  \leq \dim X$ is any number, we denote the reflexive hull of the
  $p^{\textrm{th}}$ exterior power of $\Omega^1_X$ by
  $$
  \Omega^{[p]}_X := \left( \wedge^p \Omega^1_X \right)^{**} = \imath_* \Omega^p_{X_{\reg}}.
  $$
  We refer to sections in $\Omega^{[p]}_X$ as \emph{reflexive $p$-forms on $X$}.
\end{defn}

\begin{rem}[Reflexive differentials and dualising sheaf]
  In the setting of Definition~\ref{def:reflDiff}, recall that the
  Grothendieck dualising sheaf $\omega_X$ is always reflexive. We obtain that
  $$
  \Omega^{[\dim X]}_X = \omega_X = \sO_X\bigl(K_X\bigr).
  $$
\end{rem}

\begin{notation}[Reflexive tensor operations]\label{not:relfxive}
  Let $X$ be a normal variety and $\sA$ a coherent sheaf of $\O_X$-modules, of
  rank $r$. For $n\in \bN$, set $\sA^{[n]} := (\sA^{\otimes
    n})^{**}$. Further, set $\det \sA = \bigl( \wedge^r \sA \bigr)^{**}$. If
  $\pi: X' \to X$ is a morphism of normal varieties, set $\pi^{[*]}(\sA) :=
  \bigl( \pi^*\sA \bigr)^{**}$.
\end{notation}

\subsection{Extension results for differential forms on singular spaces}
\label{ssec:extension}

One of the main ingredients for the proof of the Decomposition
Theorem~\ref{decoII} is the following extension result for differential forms on
singular spaces, recently shown in joint work of the authors and Sándor
Kovács. In its simplest form, the extension theorem asserts the following.

\begin{thm}[\protect{Extension Theorem, \cite[Thm.~1.5]{GKKP10}}]\label{thm:ext}
  Let $X$ be a quasi-projective complex algebraic variety and $D$ an effective
  $\mathbb Q$-divisor on $X$ such that the pair $(X,D)$ is Kawamata log
  terminal (klt). If $\pi: \wtilde X \to X$ is any resolution of singularities
  and $0 \leq p \leq \dim X$ any number, then the push-forward sheaf $\pi_*
  \Omega^p_{\wtilde X}$ is reflexive. \qed
\end{thm}

Using Definition~\ref{def:reflDiff}, the Extension Theorem~\ref{thm:ext} can
be reformulated, saying that $ \pi_* \Omega^p_{\tilde X} = \Omega^{[p]}_X$
for all $p \leq \dim X$.  Equivalently, if $E \subset \wtilde X$ denotes the
$\pi$-exceptional set, then the Extension Theorem asserts that for any open
set $U \subset X$, any $p$-form on $\pi^{-1}(U) \setminus E$ extends across
$E$, to give a $p$-form on $\pi^{-1}(U)$. In other words, it asserts that the
natural restriction map
$$
\Omega^p_{\wtilde X} \bigl( \pi^{-1}(U) \bigr) \to \Omega^p_{\wtilde X} \bigl(
\pi^{-1}(U) \setminus E \bigr)
$$
is surjective. We refer to the original papers \cite[Sect.~1]{GKKP10} and
\cite{GKK08} or to the survey \cite{Keb11} for an in-depth discussion.

\section{Irregularity and Albanese map of canonical varieties}
\label{sec:known}

The Albanese map is one important tool in the study of varieties with trivial
canonical divisor. The following invariant is relevant in its investigation.

\begin{defn}[Augmented irregularity]
  Let $X$ be a normal projective variety. We denote the irregularity of $X$ by
  $q(X) := h^1\bigl( X,\, \sO_X \bigr)$ and define the \emph{augmented
    irregularity} as
  $$
  \wtilde q(X) := \max \bigl\{q(\wtilde X) \mid \wtilde X \to X \ \text{a
    finite cover, étale in codimension one} \bigr\} \in \mathbb N \cup \{ \infty
  \}.
  $$
\end{defn}

\begin{rem}[Irregularity and the Albanese map]
  If $X$ is a projective variety with canonical singularities, recall from
  \cite[Sect.~8]{Kawamata85} that the Albanese map $\alpha_X : X \to \Alb(X)$ is
  well defined, and that $\dim \Alb(X) = q(X)$. Better still, if $\pi : \wtilde
  X \to X$ is any resolution of singularities, then the Albanese map
  $\alpha_{\wtilde X}$ of $\wtilde X$ agrees with $\alpha_X$. In other words,
  there exists a commutative diagram as follows,
  $$
  \xymatrix{%
    \wtilde X \ar[d]_{\pi\text{, desing.}} \ar[rrr]^{\alpha_{\wtilde
        X}}_{\text{Albanese map}} &&& \Alb(\wtilde X) \ar@{=}[d]  \\
    X \ar[rrr]^{\alpha_X}_{\text{Albanese map}} &&& \Alb(X). }
  $$
\end{rem}

The following result of Kawamata describes the Albanese map of varieties whose
canonical divisor is numerically trivial. As we will see in
Corollary~\ref{cor:qgleichnull}, this often reduces the study of varieties with
trivial canonical class to those with $\wtilde q(X) = 0$.

\begin{prop}[\protect{Fibre space structure of the Albanese map, \cite[Prop.~8.3]{Kawamata85}}]\label{Kaw}
  Let $X$ be a normal $n$-dimensional projective variety $X$ with at worst
  canonical singularities. Assume that $K_X$ is numerically trivial.  Then $K_X$
  is torsion, the Albanese map $\alpha: X \to \Alb(X)$ is surjective and has the
  structure of an étale-trivial fiber bundle.

  In other words, there exists a number $m \in \mathbb N^+$ such that
  $\sO_X(m\cdot K_X) \cong \sO_X$. Furthermore, there exists a finite étale
  cover $B \to \Alb(X)$ from an Abelian variety $B$ to $\Alb(X)$ such that the
  fiber product over $\Alb(X)$ decomposes as a product
  $$
  X \times_{\Alb(X)} B \cong F \times B,
  $$
  where $F$ is a normal projective variety. \qed
\end{prop}

\begin{rem}\label{rem:castor}
  If $X$ is a projective variety with canonical singularities and numerically
  trivial canonical class, Proposition~\ref{Kaw} implies that $q(X) = \dim
  \Alb(X) \leq \dim X$. If $\wtilde X \to X$ is any finite cover, étale in
  codimension one, then $\wtilde X$ will likewise have canonical singularities
  \cite[Prop.~5.20]{KM98} and numerically trivial canonical class. In summary,
  we see that $\wtilde q(X) \leq \dim X$. The augmented irregularity of
  canonical varieties with numerically trivial canonical class is therefore
  finite.
\end{rem}

\begin{rem}\label{rem:pollux}
  In the setting of Proposition~\ref{Kaw}, the canonical map
  $$
  F \times B \cong X \times_{\Alb(X)} B \to X
  $$
  is étale. The variety $F \times B$ is thus canonical by
  \cite[Prop.~5.20]{KM98}. Since $B$ is smooth, this automatically implies that
  $F$ is canonical. If the canonical divisor of $X$ is trivial, then $F$ will
  likewise have a trivial canonical divisor.
\end{rem}

A variant of the following corollary has appeared as \cite[Thm.~4.2]{Pe94}.

\begin{cor}[\protect{Structure of varieties with numerically trivial canonical class, cf.~\cite[Cor.~8.4]{Kawamata85}}]\label{cor:qgleichnull}
  Let $X$ be a normal $n$-dimensional projective variety with at worst canonical
  singularities. Assume that $K_X$ is numerically trivial. Then there exist
  projective varieties $A$, $Z$ and a morphism $\nu: A \times Z \to X$ such that
  the following holds.
  \begin{enumerate-p}
  \item The variety $A$ is Abelian.
  \item The variety $Z$ is normal and has at worst canonical singularities.
  \item The canonical class of $Z$ is trivial, $\omega_Z \cong \sO_Z$.
  \item The augmented irregularity of $Z$ is zero, $\wtilde q(Z) = 0$.
  \item The morphism $\nu$ is finite, surjective and étale in codimension one.
  \end{enumerate-p}
\end{cor}

\begin{proof}
  We construct a sequence of finite surjective morphisms,
  $$
  \xymatrix{ %
    F \times B \ar[rr]^\gamma_{\text{étale}} && X^{(2)}
    \ar[rrr]^\beta_{\text{étale in codim. one}} &&& X^{(1)}
    \ar[rrr]^\alpha_{\text{index-one cover}} &&& X, %
  }
  $$
  as follows. Recall from Proposition~\ref{Kaw} that the canonical divisor of
  $X$ is torsion, and let $\alpha : X^{(1)} \to X$ be the associated index-one
  cover, cf.~\cite[Sect.~5.2]{KM98} or \cite[Sect.~3.5]{Reid87}. We have seen in
  Remark~\ref{rem:castor} that the variety $X^{(1)}$ has a trivial canonical
  divisor and at worst canonical singularities.  Remark~\ref{rem:castor} also
  shows that $\wtilde q(X^{(1)})$ is finite. This implies that there exists a
  finite morphism $\beta : X^{(2)} \to X^{(1)}$, étale in codimension one, such
  that
  $$
  \wtilde q \left( X^{(1)} \right) = q \left( X^{(2)} \right) = \wtilde q
  \left( X^{(2)} \right).
  $$
  Again, $X^{(2)}$ has trivial canonical divisor and canonical singularities.
  Next, let $\gamma : F \times B \to X^{(2)}$ be the étale morphism obtained by
  applying Proposition~\ref{Kaw} and Remark~\ref{rem:pollux} to the variety
  $X^{(2)}$. The variety $B$ then satisfies the following,
  \begin{equation}\label{eq:sfoejg}
    \dim B = \dim \Alb X^{(2)} = q \left( X^{(2)} \right) = \wtilde q \left(
      X^{(2)} \right).
  \end{equation}
  Remark~\ref{rem:pollux} also asserts that $F$ has a trivial canonical divisor
  and canonical singularities.

  To finish the proof, it suffices to show that $\wtilde q(F) = 0$. If not, we
  could apply Proposition~\ref{Kaw} and Remark~\ref{rem:pollux} to $F$,
  obtaining an étale map $\delta : (F' \times B') \times B \to F \times B$,
  where $B'$ is Abelian, and of positive dimension. The composed morphism
  $\gamma \circ \delta : F' \times (B' \times B) \to X^{(2)}$ is again étale,
  showing that
  $$
  \wtilde q \left( X^{(2)} \right) \geq q\bigl( F' \times (B' \times B) \bigr)
  \geq \dim B + \dim B' > \dim B,
  $$
  thus contradicting Equation~\eqref{eq:sfoejg} above. This shows that $\wtilde
  q(F) = 0$ and finishes the proof of Corollary~\ref{cor:qgleichnull}.
\end{proof}

\section{A criterion for numerical triviality}

The proof of the Decomposition Theorem~\ref{decoII} of rests on an analysis of
the tangent sheaf of Kawamata log terminal spaces and of its destabilising
subsheaves. We will show that the determinant of any destabilising subsheaf is
trivial, at least after passing to a suitable cover. This part of the proof is
based on a criterion for numerical triviality, formulated in
Proposition~\ref{prop:angela}.  The criterion generalises the following result
which goes back to Kleiman.

\begin{lem}[Kleiman's criterion for numerical triviality]\label{lem:kleiman}
  Let $Z$ be an irreducible, normal projective variety of dimension $n \geq 2$
  and $D$ a $\mathbb Q$-Cartier divisor on $Z$. If $D \cdot H_1 \cdots H_{n-1} =
  0$ for all $(n-1)$-tuples of ample divisors $H_1, \ldots, H_{n-1}$ on $Z$,
  then $D$ is numerically trivial.
\end{lem}
\begin{proof}
  Passing to a sufficiently high multiple of $D$, we can assume without loss of
  generality that $D$ is Cartier, and thus linearly equivalent to the difference
  of two ample Cartier divisors, $D \sim H_{1,1} - H_{1,2}$. Let $H_2, \ldots,
  H_{n-1}$ be arbitrary ample divisors. Recall from \cite[Prop.~3 on page
  305]{K66} that to prove numerical triviality of $D$, it suffices to show that
  the following two equalities
  \begin{align}
    \label{eq:AY} D \cdot H_1 \cdot H_2 \cdots H_{n-1} & = 0 & \text{and}  \\
    \label{eq:BZ} D \cdot D \cdot H_2 \cdots H_{n-1} & = 0
  \end{align}
  hold. Equation~\eqref{eq:AY} holds by assumption. For Equation~\eqref{eq:BZ},
  observe that
  $$
  D \cdot D \cdot H_2 \cdots H_{n-1} = D \cdot H_{1,1} \cdot H_2 \cdots
  H_{n-1} - D \cdot H_{1,2} \cdot H_2 \cdots H_{n-1},
  $$
  where both summands are zero, again by assumption.
\end{proof}

Proposition~\ref{prop:angela}, the main result of this section, is a variant of
this criterion, adapted to the discussion of $\bQ$-factorialisations, where the
ample divisors $H_i$ are replaced by big and nef divisors which are obtained as
the pull-back of ample divisors via the $\bQ$-factorialisation map.

\begin{prop}[Criterion for numerical triviality on $\bQ$-factorialisations]\label{prop:angela}
  Let $\phi: Z' \to Z$ be a small birational morphism of irreducible, normal
  projective varieties of dimension $n \geq 2$. Let $D'$ be a pseudoeffective
  $\bQ$-Cartier $\bQ$-divisor on $Z'$ and assume that there are ample Cartier
  divisors $H_1, \ldots, H_{n-1}$ on $Z$ such that
  \begin{equation}\label{eq:intersectassumption}
    D'\cdot \phi^*(H_1) \cdots \phi^*(H_{n-1}) = 0.
  \end{equation}
  If $Z'$ is $\mathbb Q$-factorial, then $D'$ is numerically trivial.
\end{prop}

\begin{rem}[Small morphisms]
  In Proposition~\ref{prop:angela} and elsewhere in this paper, we call a
  birational morphism $\psi$ of normal, irreducible projective varieties \emph{small}
  if its exceptional set has codimension at least two.
\end{rem}

\begin{proof}[Proof of Proposition~\ref{prop:angela}]
  Let $B' \subsetneq Z'$ be the $\phi$-exceptional set, and $B := \phi(B')
  \subsetneq Z$ its image.  It suffices to prove Proposition~\ref{prop:angela}
  for a multiple of $D'$. We will therefore assume without loss of generality
  that $D'$ is an integral Cartier divisor.  To show that $D'$ is numerically
  trivial, we aim to apply Kleiman's criterion for numerical triviality,
  Lemma~\ref{lem:kleiman}. To this end, let $A_1, \ldots, A_{n-1}$ be arbitrary
  ample Cartier divisors on $Z'$. Choose numbers $a_1, \ldots, a_{n-1} \in
  \bN^+$ such that the $a_i A_i$ are very ample, choose general elements
  $\Theta_i \in |a_i A_i|$ and consider the complete intersection curve $\Gamma'
  := \Theta_1 \cap \cdots \cap \Theta_{n-1} \subsetneq Z'$. By general choice,
  the curve $\Gamma'$ is smooth and will not intersect with the small set
  $B'$. Lemma~\ref{lem:kleiman} asserts that in order to establish numerical
  triviality of $D'$ it suffices to show that
  \begin{equation}\label{eq:CZ}
    \Gamma' \cdot D' = a_1 \cdots a_{n-1} \cdot A_1 \cdots A_{n-1} \cdot D'=  0.
  \end{equation}
  To establish~\eqref{eq:CZ}, consider the image $\Gamma := \phi(\Gamma')$,
  which is a smooth curve contained in $Z \setminus B$. The curve $\Gamma$ is
  not necessarily a complete intersection curve for $H_1, \ldots, H_{n-1}$, but
  can be completed to become a complete intersection curve, as follows. Choosing
  sufficiently large numbers $m_1, \ldots, m_{n-1} \in \bN^+$, we can assume
  that the linear sub-systems
  $$
  V_i := \{ \Delta \in |m_i H_i| \,:\, \Gamma \subset \supp \Delta \}
  $$
  are positive-dimensional, basepoint-free outside of $\Gamma$, and separate
  points outside of $\Gamma$. We can also assume that the sheaves $\sI_{\Gamma}
  \otimes \sO_Z(m_iH_i)$ are spanned. Choose general elements $\Delta_i \in V_i$
  and consider the complete intersection curve
  $$
  \Gamma_{\complete} := \Delta_1 \cap \cdots \cap \Delta_{n-1} \subset Z.
  $$
  The curve $\Gamma_{\complete}$ is reduced, avoids $B$ and clearly contains
  $\Gamma$. We can thus write
  $$
  \Gamma_{\complete} = \Gamma \cup \Gamma_{\restX} \quad \text{and} \quad
  \phi^{-1}( \Gamma_{\complete} ) = \Gamma' \cup \phi^{-1}(\Gamma_{\restX}),
  $$
  where $\Gamma_{\restX}$ is an irreducible movable curve on $Z$. To end the
  argument, observe that
  \begin{align*}
    0 & = m_1 \cdots m_{n-1} \cdot D'\cdot \phi^*(H_1) \cdots \phi^*(H_{n-1})
    & \text{by Assumption~\eqref{eq:intersectassumption}}\\
    & = D' \cdot \phi^*( \Gamma_{\complete} ) = \underbrace{D' \cdot
      \Gamma'}_{\geq 0} + \underbrace{D' \cdot \phi^*(\Gamma_{\restX})}_{\geq
      0}.
  \end{align*}
  Since $D'$ is pseudoeffective, and since both $\Gamma'$ and $\phi^{-1}
  (\Gamma_{\restX})$ are movable, it follows that both summands are
  non-negative, hence zero. This shows Equation~\eqref{eq:CZ} and finishes the
  proof of Proposition~\ref{prop:angela}.
\end{proof}

\section{The tangent sheaf of varieties with trivial canonical class}
\label{sect:stability}

\subsection{Semistability of the tangent sheaf}

To prepare for the proof of the Decomposition Theorem~\ref{decoII}, we study
stability notions of the tangent sheaf. For canonical varieties with numerically
trivial canonical divisor, we show that the tangent sheaf $\sT_X$ is semistable
with respect to any polarisation, and that $\sT_X$ is stable with respect to one
polarisation if and only if it is stable with respect to any other.

The following result of Miyaoka is crucial. We will refer to
Theorem~\ref{miyaoka} at several places throughout the present paper.

\begin{thm}[Generic semipositivity, \cite{Miyaoka87, Miy85}]\label{miyaoka}
  Let $X$ be a normal projective variety of dimension $n>1$. Assume that $X$ is
  \emph{not} uniruled. Let $H_1, \ldots, H_{n-1}$ be ample line bundles on
  $X$. Then there exists a number $M \in \bN^+$ such that for all $m_1, \ldots,
  m_{n-1} > M$ the following holds.
  \begin{enumerate}
  \item The linear systems $\vert m_j H_j \vert$ are basepoint-free.
  \item If $C = D_1 \cap \cdots \cap D_{n-1} \subsetneq X$ is a curve cut out by
    general elements $D_j \in \vert m_j H_j \vert$, then $C$ is smooth, $X$ is
    smooth along $C$, and $\Omega_X^{[1]} \vert_ C$ is a nef vector bundle on
    $C$. \qed
  \end{enumerate}
\end{thm}

\begin{notation}
  The conclusion of Theorem~\ref{miyaoka} is often rephrased by saying that
  $\Omega_X^{[1]}$ is \emph{generically nef with respect to $H_1, \ldots,
    H_{n-1}$}.
\end{notation}

\begin{thm}[\protect{Mehta-Ramanathan theorem, cf.~\cite[Rem.~6.2]{MR82} and \cite{Flenner84}}]\label{thm:MR}
  In the setup of Theorem~\ref{miyaoka}, if one chooses the number $M$ large
  enough, then $\sT_X$ is semistable with respect to the polarisation $(H_1,
  \ldots, H_{n-1})$ if and only if its restriction $\sT_X|_C$ is semistable as a
  vector bundle on the curve $C$. \qed
\end{thm}

The well-known semistability of the tangent bundle is an immediate consequence
of Miyaoka's generic semipositivity result.

\begin{prop}[Semistability of the tangent sheaf]\label{prop:semistable}
  Let $X$ be a normal projective variety having at worst canonical
  singularities. If $K_X$ is numerically trivial, then $\sT_X$ is semistable
  with respect to any polarisation.
\end{prop}
\begin{proof}
  Let $(H_1, \ldots, H_{n-1})$ be arbitrary ample Cartier divisors on
  $X$. Choose numbers $M, m_1, \ldots, m_{n-1}$ and construct a general complete
  intersection curve $C \subsetneq X$ as in Theorems~\ref{miyaoka} and
  \ref{thm:MR}.

  It follows from numerical triviality of $K_X$ and from the assumption on the
  singularities that $X$ is not uniruled. Theorem~\ref{miyaoka} therefore
  asserts that the restriction $\Omega_X^{[1]}|_C$ is nef.  Since additionally
  $K_X \cdot C = 0$ by assumption, it follows that the bundle $\sT_X|_C$ is
  semistable. The Mehta-Ramanathan Theorem~\ref{thm:MR} then shows that $\sT_X $
  is semistable with respect to $(H_1, \ldots, H_{n-1})$.
\end{proof}

\begin{rem}
  Although semistable, the tangent sheaf of a variety with trivial canonical
  bundle might not be stable. To give an easy example, the tangent sheaf of the
  product of two such varieties is not stable.
\end{rem}

\subsection{Pseudoeffectivity of quotients of $\Omega^{[p]}_X$}
\label{ssec:pseff}

The proof of our main result uses the following criterion for the
pseudoeffectivity of Weil divisors on $\bQ$-factorial spaces. In case where $X$
is smooth, this has been shown by Campana-Peternell.

\begin{prop}[\protect{Pseudoeffectivity of quotients of $\Omega_X^{[p]}$, cf.~\cite[Thm.~0.1]{CP11}}]\label{prop:pseudoeffectivitygeneralised}
  Let $X$ be a normal, $\bQ$-factorial projective variety and let $D$ be a Weil
  divisor on $X$. Assume that there exists a number $0 \leq p \leq \dim X$ and a
  non-trivial sheaf morphism $\psi: \Omega_X^{[p]} \to \sO_X(D)$. If $X$ is not
  uniruled, then $D$ is pseudoeffective.
\end{prop}
\begin{proof}
  Assume that $X$ is not uniruled, and that a non-trivial sheaf morphism $\psi:
  \Omega_X^{[p]} \to \sO_X(D)$ is given. The existence of a resolution of
  singularities combined with a classical result of Rossi
  \cite[Thm.~3.5]{Rossi68} shows that there exists a strong log resolution of
  singularities $\pi: \wtilde X \to X$ with the additional property that
  $\pi^{[*]} \sO_{ X}(D)$ is locally free. The $\pi$-exceptional set $E
  \subsetneq \wtilde X$ is then of pure codimension one, and has simple normal
  crossing support. Finally, write $\pi^{[*]} \sO_{X}(D) = \sO_{\wtilde
    X}(\wtilde D)$, where $\wtilde D$ is a divisor on $\wtilde X$ that agrees
  with the strict transform $\pi_*^{-1}(D)$ outside of the $\pi$-exceptional set
  $E$.

  Since the two sheaves $\pi^{[*]}\Omega_X^{[p]}$ and $\Omega_{\wtilde X}^p$ are
  isomorphic outside of $E$, there exists a number $m \in \bN^+$ and a sheaf
  morphism
  $$
  \wtilde \psi : \Omega_{\wtilde X}^p \to \sO_{\wtilde X}(\wtilde D +mE)
  $$
  which agrees on $\wtilde X \setminus E$ with the pull-back of $\psi$. If $\sL
  := (\Image \wtilde \psi)^{**}$ denotes the reflexive hull of the image sheaf,
  then $\sL$ is invertible by \cite[Lem.~1.1.15]{OSS}. Recalling that
  $\Omega_{\wtilde X}^p$ is a quotient of $(\Omega_{\wtilde X}^1)^{\otimes p}$,
  it follows from \cite[Thm.~0.1]{CP11} that the line bundle $\sL$ is in fact
  pseudoeffective. It follows that $\wtilde D+mE$ is pseudoeffective as well,
  since it contains the pseudoeffective subsheaf $\sL$.

  Since $X$ is assumed to be $\bQ$-factorial, it is clear that the
  cycle-theoretic push-forward of any pseudoeffective divisor on $\wtilde X$ is
  $\bQ$-Cartier and again pseudoeffective. Using that $\pi_* \wtilde D = D$,
  this shows our claim.
\end{proof}

\subsection{Stability of the tangent sheaf}
\label{ssec:xxa}

While Theorem~\ref{miyaoka} and Proposition~\ref{prop:semistable} are fairly
standard today, the following result, which shows that stability of
the tangent bundle is independent of the chosen polarisation, is new
to the best of our knowledge. We feel that it might be of independent interest.

\begin{prop}[Stability of $\sT_X$ does not depend on polarisation]\label{prop:indepSS}
  Let $X$ be a normal projective variety having at worst canonical
  singularities. Assume that $K_X$ is numerically equivalent to zero. Let $h_1
  := (H_1^{(1)}, \ldots, H_{n-1}^{(1)})$ and $h_2 := (H_1^{(2)}, \ldots,
  H_{n-1}^{(2)})$ be two sets of ample polarisations.  If $\sT_X$ is
  $h_1$-stable, then it is also $h_2$-stable.
\end{prop}
\begin{proof}
  Suppose that $\sT_X$ is \emph{not} $h_2$-stable. Even though not stable,
  recall from Proposition~\ref{prop:semistable} that $\sT_X$ is semistable with
  respect to $h_2$. Choose a Jordan-Hölder filtration of $\sT_X$ with respect to
  $h_2$ and let $0 \varsubsetneq \sS \varsubsetneq \sT_X$ be its first term,
  cf.~\cite[Sect.~1.5]{HL97}. The sheaf $\sS$ is then $h_2$-stable, and
  saturated as a subsheaf of $\sT_X$. Semistability of $\sT_X$ implies that
  $\sS$ has $h_2$-slope zero. In other words, $\mu_{h_2}(\sS) = 0$.
  
  Next, let $\phi: X' \to X$ be a $\bQ$-factorialisation, that is, a small
  birational morphism where $X'$ is $\bQ$-factorial and has at worst canonical
  singularities. The existence of $\phi$ is established in
  \cite[Lem.~10.2]{BCHM06}. Consider the reflexive pull-back sheaf $\sS' :=
  \phi^{[*]} \sS$. Since $\sS'$ injects into $\sT_{X'}$ outside of the small
  $\pi$-exceptional set, and since both sheaves are reflexive, we obtain an
  injection $ \sS' \into \sT_{X'} = \phi^{[*]} \sT_X$.

  We claim that $\det \sS'$ is numerically trivial on $X'$. To this end, recall
  from Proposition~\ref{prop:pseudoeffectivitygeneralised} that $\bigl( \det
  \sS' \bigr)^*$ is pseudoeffective. On the other hand, since $\sS'$ and
  $\phi^*(\sS)$ agree outside of the $\pi$-exceptional set, we obtain
  $$
  \bigl( \det \sS' \bigr)^* \cdot \phi^*\bigl( H_1^{(2)} \bigr) \cdots \phi^*\bigl(
  H_{n-1}^{(2)} \bigr) = -\mu_{h_2}(\sS) = 0.
  $$
  Together with Proposition~\ref{prop:angela}, these two observations show that
  $(\det \sS')^*$ and $\det \sS'$ are numerically trivial, as claimed.

  Using numerical triviality of $\det \sS'$, the same line of reasoning now
  gives
  $$
  \mu_{h_1}(\sS) = \bigl( \det \sS' \bigr) \cdot \phi^*\bigl( H_1^{(1)} \bigr)
  \cdots \phi^*\bigl( H_{n-1}^{(1)} \bigr) = 0,
  $$
  showing that $\sT_X$ is \emph{not} $h_1$-stable. This contradiction concludes
  the proof.
\end{proof}

\section{Differential forms on varieties with trivial canonical class}
\label{sec:kawamata}

\subsection{Non-degeneracy of the wedge product}
\label{sec:nondegen}

Differential forms on smooth varieties with trivial canonical class were studied
by Bogomolov \cite{Bogomolov74}.  In this section we apply the Extension Theorem~\ref{thm:ext}
to study reflexive differential forms on singular varieties with trivial
canonical classes, following an approach discussed in \cite{Pe94}. The results
obtained in this section will play an important role in the proof of the
Decomposition Theorem~\ref{decoII}, which is given in the subsequent
Section~\ref{sect:deco}.

\begin{prop}[Non-degeneracy of the wedge product]\label{prop:forms-1}
  Let $X$ be a normal $n$-dimensional projective variety $X$ having at worst
  canonical singularities. Suppose that $\omega_X \cong \sO_X$. If $0 \leq p
  \leq n$ is any number, then the natural pairing given by the wedge product,
  $$
  \bigwedge : H^0\bigl( X,\, \Omega^{[p]}_X \bigr) \times H^0\bigl( X,\,
  \Omega^{[n-p]}_X \bigr) \longrightarrow H^0\bigl( X,\, \omega_X \bigr) \cong
  \bC,
  $$
  is non-degenerate.
\end{prop}

\begin{rem}
  If $\imath: X_{\reg} \hookrightarrow X$ is the inclusion of the smooth locus
  into $X$, recall from Definition~\ref{def:reflDiff} that $\Omega^{[p]}_X =
  \imath_* \Omega^p_{X_{\reg}}$.  Given a non-zero form $\eta \in
  H^0\bigl(X_{\reg},\, \Omega^p_{X_{\reg}}\bigr)$,
  Proposition~\ref{prop:forms-1} simply says that there exists a
  ``complementary'' form $\phi \in H^0\bigl( X_{\reg},\,
  \Omega^{n-p}_{X_{\reg}}\bigr)$, such that $\eta \wedge \phi$ extends to a
  non-zero, hence nowhere vanishing section of $\sO_X \cong \omega_X \cong
  \Omega_X^{[n]}$.
\end{rem}

\begin{rem}
  Proposition~\ref{prop:forms-1} has been shown in relevant cases in
  \cite[Prop.~5.8]{Pe94}. Our proof of Proposition~\ref{prop:forms-1} follows
  \cite{Pe94} closely.
\end{rem}

\subsubsection{Proof of Proposition~\ref*{prop:forms-1}}

We end the present Section~\ref{sec:nondegen} with a proof of
Proposition~\ref{prop:forms-1}. To improve readability, the proof is subdivided
into six, mostly independent steps.

\subsubsection*{Step 1 in the proof of Proposition~\ref*{prop:forms-1}: Setup of notation} 

We choose a desingularisation $\pi: \wtilde X \to X$ of $X$. Denote the
$\pi$-exceptional set by $E \subset \wtilde X$ and fix a non-zero section
$\sigma \in H^0\bigl(X, \omega_X \bigr)$. Since $\omega_X$ is invertible, and
since $X$ has canonical singularities, it follows immediately from the
definition that the pull-back of $\sigma$ is a holomorphic $n$-form on $\wtilde
X$, possibly with zeroes along the exceptional set, say
$$
\tau := \pi^*(\sigma) \in H^0\bigl(\wtilde X,\, \omega_{\wtilde X} \bigr).
$$
Because $H^0\bigl(\wtilde X,\, \omega_{\wtilde X} \bigr) \cong H^0\bigl(X,\,
\omega_{X} \bigr) = \C \cdot \sigma$, the form $\tau$ clearly spans the vector
space $H^0\bigl(\wtilde X,\, \omega_{\wtilde X} \bigr)$. By the Extension
Theorem~\ref{thm:ext} we have $\pi_* \Omega_{\wtilde X}^p = \Omega_X^{[p]}$. To
prove Proposition~\ref{prop:forms-1}, it is therefore sufficient to prove the
following claim.

\begin{claim}\label{claim:prop61}
  Given any holomorphic $p$-form $\alpha \in H^0\bigl(\wtilde X, \,
  \Omega_{\wtilde X}^p \bigr)$ there exists a ``complementary'' form $\beta \in
  H^0\bigl(\wtilde X, \, \Omega_{\wtilde X}^{n-p} \bigr)$ such that $\alpha
  \wedge \beta = \tau$.
\end{claim}

\stepcounter{thm}

\subsubsection*{Step 2 in the proof of Proposition~\ref*{prop:forms-1}: Dolbeault cohomology on $\wtilde X$}

Following standard notation, let $\mathcal A^{a,b}$ denote the sheaf of
$\bC$-valued differentiable forms of type $(a,b)$ on $\wtilde X$. Taking
products and wedge products with $\sigma$ and $\tau$, respectively, we obtain
sheaf morphisms,
$$
\begin{array}{rcccrccc}
  \psi_X : & \sO_X            & \to     & \omega_X         & \qquad \psi_{\wtilde X} : & \mathcal \sO_{\wtilde X} & \to     & \omega_{\wtilde X} \\
           & f                & \mapsto & f \cdot \sigma   &                           & f                        & \mapsto & f \cdot \tau       \\[5mm]
  \psi_0 : & \mathcal A^{0,0} & \to     & \mathcal A^{n,0} & \qquad \psi_q :           & \mathcal A^{0,q}         & \to     & \mathcal A^{n,q} \\
           & f                & \mapsto & f \cdot \tau     &                           & \alpha                   & \mapsto & \alpha \wedge \tau.  
\end{array}
$$
where $0 < q \leq n$. Observe that $\psi_X$ is isomorphic by assumption.

Since $\tau$ is holomorphic, its exterior derivative vanishes
$\overline \partial \tau = 0$. This immediately implies relations
\begin{equation}\label{eq:comm}
  \overline \partial \circ \psi_q = \psi_{q+1} \circ \overline \partial \text{\quad for all $0 \leq q \leq n$.}
\end{equation}
In particular, the sheaf morphisms $\psi_q$ induce well-defined morphisms
between Dolbeault cohomology groups,
$$
\phi_q : H^{0,q}\bigl( \wtilde X \bigr) \to H^{n,q}\bigl( \wtilde X \bigr)
\text{\quad for all $0 \leq q \leq n$.}
$$
We will later see in Step~4 of this proof that the morphisms $\phi_q$ are in
fact isomorphic.

\subsubsection*{Step 3 in the proof of Proposition~\ref*{prop:forms-1}: Dolbeault and sheaf cohomology on $\wtilde X$}

The sheaf morphisms $\psi_X$ and $\psi_{\wtilde X}$ induce additional morphisms
between sheaf cohomology groups,
\begin{align*}
H^q\bigl( \psi_X \bigr) & : H^q\bigl( X,\, \sO_X \bigr) \to H^q\bigl( X,\, \omega_X \bigr) & \text{and}\\
H^q\bigl( \psi_{\wtilde X} \bigr) & : H^q\bigl( \wtilde X,\, \sO_{\wtilde X} \bigr) \to H^q\bigl( \wtilde X,\, \omega_{\wtilde X} \bigr),
\end{align*}
for all $0 \leq q \leq n$. Again, observe that the morphisms $H^q\bigl( \psi_X
\bigr)$ are isomorphic by assumption. We will see in Step~4 of this proof that
the morphisms $H^q\bigl( \psi_{\wtilde X} \bigr)$ are isomorphic as well.

The morphisms $\phi_q$ and $H^q\bigl( \psi_{\wtilde X} \bigr)$ are closely
related. In fact, it follows from \eqref{eq:comm} that the sheaf morphisms
$\psi_\bullet$ align to give a morphism between the Dolbeault resolutions of
$\sO_{\wtilde X}$ and $\omega_{\wtilde X}$, respectively. In other words, there
exists a commutative diagram as follows,
\begin{equation}\label{eq:morDolbRes}
  \begin{split}
    \xymatrix{%
      \sO_{\wtilde X} \ar@{^(->}[r] \ar[d]_{\psi_{\widetilde X}} & \mathcal A^{0,0} \ar[r]^{\overline \partial}  \ar[d]_{\psi_0} & \mathcal A^{0,1} \ar[r]^{\overline \partial}  \ar[d]_{\psi_1} & \mathcal A^{0,2} \ar[r]^{\overline \partial} \ar[d]_{\psi_2} & \mathcal A^{0,3} \ar[r]^{\overline \partial} \ar[d]_{\psi_3} & \cdots \\
      \omega_{\wtilde X} \ar@{^(->}[r] & \mathcal A^{n,0} \ar[r]_{\overline \partial} & \mathcal A^{n,1} \ar[r]_{\overline \partial} & \mathcal A^{n,2} \ar[r]_{\overline \partial} & \mathcal A^{n,3} \ar[r]_{\overline \partial} & \cdots \\
    }
  \end{split}
\end{equation}
The following is then a standard consequence of homological algebra, see for
instance \cite[Ch.~IV, \S 6, eq.~(6.3)]{DemaillyBook}.

\begin{conclusion}\label{conc:uSQ}
  There exist commutative diagrams
  $$
  \xymatrix{%
    H^{0,q}\bigl( \wtilde X\bigr) \ar[rr]^{\phi_q} \ar[d]_{\text{Dolbeault isom.}}^{\cong} && H^{n,q}\bigl( \wtilde X\bigr) \ar[d]^{\text{Dolbeault isom.}}_{\cong} \\
    H^q\bigl( \wtilde X,\, \sO_{\wtilde X} \bigr) \ar[rr]^{H^q(\psi_{\wtilde X})} && H^q \bigl( \wtilde X,\, \omega_{\wtilde X} \bigr)
  }
  $$
  for all indices $0 \leq q \leq n$. \qed
\end{conclusion}

\subsubsection*{Step 4 in the proof of Proposition~\ref*{prop:forms-1}: cohomology on $\wtilde X$ and on  $X$}

Next, we aim to compare cohomology groups on $\wtilde X$ with those on $X$. More
precisely, we claim the following.

\begin{claim}\label{claim:lSQ}
  Given any index $0 \leq q \leq n$, there exist morphisms $\rho_{\sO}$,
  $\rho_{\omega}$ forming a commutative diagram as follows,
  \begin{equation}\label{eq:lSQ}
    \begin{split}
      \xymatrix{%
        H^q\bigl( \wtilde X,\, \sO_{\wtilde X} \bigr) \ar[rr]^{H^q(\psi_{\wtilde X})} && H^q \bigl( \wtilde X,\, \omega_{\wtilde X} \bigr) \\
        H^q\bigl( X,\, \sO_X \bigr) \ar[rr]^{H^q(\psi_X)}_{\cong} \ar[u]^{\rho_{\sO}}_{\cong} && H^q \bigl( X,\, \omega_X \bigr). \ar[u]_{\rho_{\omega}}^{\cong}
      }
    \end{split}
  \end{equation}
  In particular, the morphisms $H^q(\psi_{\wtilde X})$ are isomorphic for all $0
  \leq q \leq n$.
\end{claim}
\begin{proof}
  Since $\pi$ has connected fibers, and since $X$ has only canonical
  singularities, we have canonical identifications
  $$
  \pi_* \sO_{\wtilde X} \cong \sO_X \qquad \text{and} \qquad \pi_* \omega_{\wtilde X} \cong \omega_X.
  $$
  Observe that the section $\sigma$, seen as a section in $\pi_* \omega_{\wtilde
    X}$, will be identified with the differential form $\tau$. Using these
  identifications, we need to show that there exist two morphisms
  $$
  \rho_{\sO} : H^q\bigl( X,\, \sO_X \bigr) \to H^q\bigl( \wtilde X,\, \sO_{\wtilde
    X} \bigr) \quad\text{and}\quad \rho_{\omega} : H^q \bigl( X,\, \omega_X \bigr)
  \to H^q \bigl( \wtilde X,\, \omega_{\wtilde X} \bigr),
  $$
  which make Diagram~\eqref{eq:lSQ} commutative and are isomorphic. While this
  can be concluded from universal properties and spectral sequences, we found it
  more instructive to give an elementary construction using \v{C}ech cohomology.

  To this end, choose an open affine cover $(U_i)_{i \in I}$ of $X$, which will
  be acyclic for any coherent sheaf, and let $\rho_{\sO}$, $\rho_{\omega}$ be
  the compositions of the vertical arrows in the following natural diagram.
  $$
  \xymatrix{ %
    \check H^q\bigl( \wtilde X ,\, \sO_{\wtilde X} \bigr)
    \ar[rrr]^{\check H^q(\psi_{\wtilde X})} &&& 
    \check H^q\bigl( \wtilde X ,\, \omega_{\wtilde X} \bigr) \\
    \check H^q\bigl( (\pi^{-1}U_i)_{i \in I} ,\, \sO_{\wtilde X} \bigr) 
    \ar[rrr]^{\check H^q((\pi^{-1}U_i)_{i \in I} ,\,\psi_{\wtilde X})} \ar[u]^{r_{\wtilde X, \sO}}  \ar@{<->}[d]_{\cong} &&& 
    \check H^q\bigl( (\pi^{-1}U_i)_{i \in I} ,\, \omega_{\wtilde X} \bigr)
    \ar[u]_{r_{\wtilde X, \omega}} \ar@{<->}[d]^{\cong} \\
    \check H^q\bigl( (U_i)_{i \in I} ,\, \pi_*\sO_{\wtilde X} \bigr) 
    \ar[rrr]^{\check H^q((U_i)_{i \in I} ,\, \pi_* \psi_{\wtilde X})} \ar@{<->}[d]_{\cong} &&&
    \check H^q\bigl( (U_i)_{i \in I} ,\, \pi_* \omega_{\wtilde X} \bigr) 
    \ar@{<->}[d]^{\cong} \\
    \check H^q\bigl( (U_i)_{i \in I} ,\, \sO_X \bigr) 
    \ar[rrr]^{\check H^q((U_i)_{i \in I} ,\,\psi_X)} \ar[d]_{r_{X, \sO}}^{\cong} &&&
    \check H^q\bigl( (U_i)_{i \in I} ,\, \omega_X \bigr) 
    \ar[d]^{r_{X, \omega}}_{\cong} \\
    \check H^q\bigl( X ,\, \sO_X \bigr)
    \ar[rrr]^{\check H^q(\psi_X)}_{\cong} \ar@/^0.7cm/@<1.2cm>[uuuu]^{\check \rho_{\sO}} &&&
    \check H^q\bigl( X ,\, \omega_X \bigr) 
    \ar@/_0.7cm/@<-1.2cm>[uuuu]_{\check \rho_{\omega}}
  }
  $$
  Here, the morphisms $r_{\bullet, \bullet}$ are the standard refinement
  morphisms that map \v{C}ech cohomology groups defined with respect to a
  specific open covering into \v{C}ech cohomology. Identifying \v{C}ech and
  sheaf cohomology, commutativity of Diagram~\eqref{eq:lSQ} is then immediate.
  
  To prove that $\rho_{\sO}$ and $\rho_{\omega}$ are isomorphic, it suffices to
  show that the refinement morphisms, $r_{\wtilde X, \sO}$ and $r_{\wtilde X,
    \omega}$ are isomorphic. We do that by showing that the covering $(\pi^{-1}
  U_i)_{i \in I}$ is acyclic for both $\sO_{\wtilde X}$ and $\omega_{\wtilde
    X}$. That, however, follows immediately from the following two well-known
  vanishing results which hold for all indices $q > 0$,
  \begin{align*}
  R^q \pi_* \sO_{\wtilde X} & = 0 && \text{because $X$ has rational singularities, \cite[Thm.~5.22]{KM98}} \\
  R^q \pi_* \omega_{\wtilde X} & = 0 && \text{by Grauert-Riemenschneider vanishing, \cite[Cor.~2.68]{KM98}.}
  \end{align*}
  The finishes the proof of Claim~\ref{claim:lSQ}.
\end{proof}

Combining Conclusion~\ref{conc:uSQ} and Claim~\ref{claim:lSQ}, we arrive at the
following statement, which summarises the results obtained so far.

\begin{conclusion}\label{conncl:67}
  The morphisms $\phi_q : H^{0,q}\bigl( \wtilde X \bigr) \to H^{n,q}\bigl(
  \wtilde X \bigr)$ are isomorphic for all $0 \leq q \leq n$. \qed
\end{conclusion}

\subsubsection*{Step 5 in the proof of Proposition~\ref*{prop:forms-1}: Serre duality} 

Given any index $0 \leq p \leq n$, consider the complex bilinear form $\rho$
obtained as the composition of the following morphisms,
\begin{align*}
  H^0 \bigl(\wtilde X,\, \Omega^p_{\wtilde X} \bigr) \times H^0 \bigl(\wtilde X,\, \Omega^{n-p}_{\wtilde X} \bigr) 
  & \xrightarrow{\makebox[2.5cm]{\scriptsize Dolbeault isom.}}
  H^{p,0}(\wtilde X) \times H^{n-p,0}(\wtilde X)\\
  & \xrightarrow{\makebox[2.5cm]{\scriptsize $\Id \times$ conjugation}} 
  H^{p,0}(\wtilde X) \times H^{0,n-p}(\wtilde X)\\
  & \xrightarrow{\makebox[2.5cm]{\scriptsize $\Id \times \phi_{n-p}$}} 
  H^{p,0}(\wtilde X) \times H^{n,n-p}(\wtilde X)\\
  & \xrightarrow{\makebox[2.5cm]{\scriptsize $\Id \times$ conjugation}} 
  H^{p,0}(\wtilde X) \times H^{n-p,n}(\wtilde X)\\
  & \xrightarrow{\makebox[2.5cm]{\scriptsize $s$}} 
  \C,
\end{align*}
where $s$ is the perfect pairing given by Serre duality, cf.~\cite[Ch.~VI,
Thm.~7.3]{DemaillyBook}. 

Recall from Conclusion~\ref{conncl:67} that with the exception of $s$ all maps used in the definition of
$\rho$ are isomorphisms. It follows that $\rho$ is a
perfect pairing. Unwinding the definition, $\rho$ is given in elementary terms
as follows,
$$
\rho : H^0 \bigl(\wtilde X,\, \Omega^p_{\wtilde X} \bigr) \times H^0 \bigl(\wtilde X,\,
\Omega^{n-p}_{\wtilde X} \bigr) \to \C, \qquad \bigl( \alpha, \beta \bigr)
\mapsto \int_X \alpha \wedge \beta \wedge \overline \tau.
$$

\subsubsection*{Step 6 in the proof of Proposition~\ref*{prop:forms-1}: End of proof} 

We are now ready to prove Claim~\ref{claim:prop61}. Assume we are given a
non-zero form $\alpha \in H^0\bigl(\wtilde X, \, \Omega_{\wtilde X}^p
\bigr)$. Using that $\rho$ is a perfect pairing, we can therefore find a form
$\beta \in H^0\bigl(\wtilde X, \, \Omega_{\wtilde X}^{n-p} \bigr)$ such that
\begin{equation}\label{eq:intFRM}
  \rho(\alpha, \beta) = \int_X \alpha \wedge \beta \wedge \overline \tau  = 1
\end{equation}
Equation~\eqref{eq:intFRM} implies that $\alpha \wedge \beta$ is a non-vanishing
element of $H^0 \bigl( \wtilde X,\, \omega_{\wtilde X} \bigr)$. Since $H^0
\bigl( \wtilde X,\, \omega_{\wtilde X} \bigr)$ is one-dimensional, there exists
a scalar $\lambda \in \bC^*$ such that 
$$
\tau = \lambda \cdot (\alpha \wedge \beta) = \alpha \wedge (\lambda \cdot \beta).
$$
This finishes the proof of Claim~\ref{claim:prop61} and hence of
Proposition~\ref{prop:forms-1}. \qed

\subsection{Hodge duality for klt spaces}

If $X$ is a projective manifold, Hodge theory gives a complex-linear isomorphism
between the spaces $H^0\bigl(X,\, \Omega^{p}_X\bigr)$ and $\overline{
  H^p\bigl(X,\, \sO_X\bigr) }$. We show that the same statement holds for
reflexive differentials if $X$ has canonical singularities, or more generally if
$X$ is the base space of a klt pair.

\begin{prop}[Hodge duality for klt spaces]\label{prop:forms-2}
  Let $X$ be a normal $n$-dimensional projective variety $X$. Suppose that there
  exists an effective $\bQ$-divisor $D$ on $X$ such that $(X,D)$ is klt. Given
  any number $0 \leq p \leq n$, there are complex-linear isomorphisms
  $$
  H^0\bigl(X,\, \Omega^{[p]}_X\bigr) \cong H^0 \left(X_{\reg},
  \,\Omega^p_{X_{\reg}} \right) \cong \overline{ H^p\bigl(X,\, \sO_X\bigr) }.
  $$
\end{prop}
\begin{proof}
  Fix a resolution of singularities $\pi: \wtilde X \to X$ of $X$. Then we have
  the following chain of complex-linear isomorphisms
  \begin{align*}
    H^0\bigl(X, \,\Omega_X^{[p]}\bigr) & \cong H^0\bigl(\wtilde X,\, \Omega^p_{\wtilde X}\bigr)  && \text{Since $\pi_*\Omega^p_{\wtilde X} = \Omega^{[p]}_X$ by the Extension Theorem~\ref{thm:ext}}\\
    &\cong  H^{p,0}\bigl(\wtilde X\bigr) && \text{Dolbeault isomorphism}\\
    &\cong \overline{H^{0,p}\bigl(\wtilde X\bigr)} &&  \text{Conjugation}\\
    & \cong \overline{H^p\bigl(\wtilde X,\, \sO_{\wtilde X}\bigr)} && \text{Dolbeault isomorphism}\\
    & \cong \overline{H^p\bigl(X,\,\sO_X\bigr)} && \text{$X$ has rational singularities, \cite[Thm.~5.22]{KM98}.}\qedhere
  \end{align*}
\end{proof}

We list a few immediate consequences of the results obtained so far.

\begin{cor}\label{cor:64}
  Let $X$ be a normal $n$-dimensional projective variety $X$ having at worst
  canonical singularities.  Suppose that the canonical sheaf of $X$ is trivial,
  $\omega_X \cong \sO_X$. Then the following holds.
  \begin{enumerate}
  \item Non-zero forms $\eta \in H^0 \bigl(X_{\reg},\,
    \Omega^q_{X_{\reg}}\bigr)$ do not have any zeroes.
  \item\label{il:trout} For all $0 \leq p \leq n$, we have complex-linear
    isomorphisms $H^0\bigl(X,\, \Omega_X^{[p]}\bigr) \cong H^0\bigl(X, \,
    \Omega^{[n-p]}_X \bigr)^*$, canonically given up to multiplication with a
    constant.
  \item If the dimension of $X$ is odd, then $\chi\bigl(X,\,\sO_X\bigr) =
    0$. \qed
  \end{enumerate}
\end{cor}

\begin{cor}[Existence of forms on canonical varieties with  $K_X \equiv 0$]\label{cor:max1}
  Let $X$ be a normal $n$-dimensional projective variety $X$ having at worst
  canonical singularities. Assume that $\wtilde q(X) = 0$ and that the canonical
  divisor $K_X$ is numerically trivial.  Then
  $$
  h^0 \bigl(X,\, \Omega^{[1]}_X\bigr) = h^0 \bigl(X,\, \Omega^{[n-1]}_X\bigr) =
  0.
  $$
\end{cor}
\begin{proof}
  We show that $H^0 \bigl(X,\, \Omega^{[n-1]}_X\bigr) = 0$. Assume to the
  contrary, and let $\sigma$ be a non-zero reflexive $(n-1)$-form on $X$.
  Recalling from Kawamata's analysis of the Albanese map, Proposition~\ref{Kaw},
  that $K_X$ is torsion, let $f : \wtilde X \to X$ be the associated index-one
  cover. The morphism $f$ is finite and étale in codimension one, the space
  $\wtilde X$ has canonical singularities, and trivial canonical sheaf
  $\omega_{\wtilde X} \cong \sO_{\wtilde X}$, cf.~\cite[5.19 and 5.20]{KM98}. In
  particular, the reflexive form $\sigma$ pulls back to a give non-vanishing
  reflexive $(n-1)$-form $\wtilde \sigma$ on $\wtilde X$. Furthermore, observe
  that the covering space $\wtilde X$ satisfies all requirements made in
  Proposition~\ref{prop:forms-2} and Corollary~\ref{cor:64}. This shows
  $$
  0 =\wtilde q(X) \geq q(\wtilde X) = h^1 \bigl( \wtilde X,\, \sO_{\wtilde X} \bigr) \underset{\text{Prop.~\ref{prop:forms-2}}}{=}
  h^0 \bigl( \wtilde X,\, \Omega^{[1]}_{\wtilde X} \bigr) \underset{\text{Cor.~(\ref{cor:64}.\ref{il:trout})}}{=} h^0 \bigl( \wtilde
  X,\, \Omega^{[n-1]}_{\wtilde X} \bigr),
  $$
  contradicting the existence of $\wtilde \sigma$. The same argument also shows
  $H^0 \bigl(X,\, \Omega^{[1]}_X\bigr) = 0$, finishing the proof of
  Corollary~\ref{cor:max1}.
\end{proof}

\subsection{Existence of complementary sheaves}
\label{ssec:compSheaf}

We conclude the present Section~\ref{sec:kawamata} with a final corollary which
generalises \cite[Lem.~5.11]{Pe94}; see also \cite[p.~581]{Bogomolov74}. It
shows that saturated subsheaves of $\sT_X$ with trivial determinant often have a
complementary subsheaf which presents $\sT_X$ as a direct
product. Corollary~\ref{cor:split} is thus an important ingredient in the proof
of our main result, the Decomposition Theorem~\ref{decoII}.

\begin{cor}[Existence of complementary subsheaves in $\sT_X$]\label{cor:split}
  Let $X$ be a normal projective variety with trivial canonical sheaf $\omega_X
  \cong \sO_X$, having at worst canonical singularities. Let $\sE \subsetneq
  \sT_X$ be a saturated subsheaf with trivial determinant, $\det \sE \cong
  \sO_X$. Then there exists a subsheaf $\sF \subsetneq \sT_X$ with trivial
  determinant such that
  $$
  \sT_X \cong\sE \oplus \sF.
  $$
\end{cor}

We will prove Corollary~\ref{cor:split} in the remainder of the present
Section~\ref{ssec:compSheaf}. For convenience, the proof is subdivided into four
steps.

\subsubsection*{Step 1 in the proof of Corollary~\ref*{cor:split}: Setup}

We consider the obvious quotient sequence
\begin{equation}\label{eq:quot}
  0 \longrightarrow \sE \xrightarrow{\quad\alpha\quad} \sT_X 
  \xrightarrow{\quad\beta\quad} \underbrace{\sT_X/\sE}_{=:\sQ} \longrightarrow 0.
\end{equation}
Since $\sE$ is saturated in the reflexive sheaf $\sT_X$, it is itself
reflexive. Further, the associated quotient $\sQ$ is a torsion free sheaf, say
of rank $r > 0$. We aim to split sequence~\eqref{eq:quot} in codimension one. To
be more precise, let $Z \subsetneq X$ be the smallest set such that $X^\circ :=
X \setminus Z$ is smooth and $\sQ|_{X^\circ}$ is locally free. Since $X$ is
normal, and since torsion-free sheaves on manifolds are locally free in
codimension one, \cite[p.~148]{OSS}, it follows that $Z$ is small, that is,
$\codim_X Z \geq 2$. If we can find a splitting of Sequence~\eqref{eq:quot} on
$X^\circ$ and write $\sT_{X^\circ} \cong \sE|_{X^\circ} \oplus \sQ|_{X^\circ}$,
it will follow from reflexivity that $\sT_X \cong \sE \oplus \sQ^{**}$, and the
proof of Corollary~\ref{cor:split} will be finished.

\subsubsection*{Step 2 in the proof of Corollary~\ref*{cor:split}: Construction of the splitting}

In order to construct the splitting, recall the assumptions that $\det \sE \cong
\sO_X$ and $\omega_X \cong \sO_X$. As a consequence, we have triviality of
determinants, $\det \sQ \cong \det \sQ^* \cong \sO_X$, see~\cite[Ch.~V,
Prop.~6.9]{Kob87} for details. Let $\eta_{\sQ} \in H^0 \bigl(X,\, \det
\sQ^*\bigr)$ be any non-vanishing section.

Taking duals on $X^\circ$, Sequence~\eqref{eq:quot} gives injections
$$
\beta^* : \sQ^*|_{X^\circ} \to \Omega^1_{X^\circ}, \quad \wedge^r\beta^* :
\wedge^r \sQ^*|_{X^\circ} \to \Omega^r_{X^\circ} \quad\text{and}\quad \det
\beta^* : \det \sQ^* \to \Omega^{[r]}_X.
$$
We obtain a non-trivial reflexive form
$$
\eta := \bigl( \det \beta^* \bigr) (\eta_{\sQ}) \in H^0\bigl(X,\,
\Omega_X^{[r]}\bigr) \setminus \{ 0\}.
$$
Denoting the dimension of $X$ by $n$, Proposition~\ref{prop:forms-1} asserts the
existence of a complementary reflexive form $\mu \in H^0\bigl(X,
\,\Omega^{[n-r]}_X\bigr)$ such that $\eta \wedge \mu$ gives a nowhere-vanishing
section of $\omega_X$. The triviality of $\det \sQ|_{X^\circ}$ and of
$\omega_{X^\circ} = \det \sT_{X^\circ}^*$ thus gives isomorphisms of sheaves,
\begin{equation}\label{eq:7so}
  \begin{array}{rccccrccc}
    \delta_{\sQ} : & \sQ|_{X^\circ} & \to & \wedge^{r-1} \sQ^*|_{X^\circ} \\[1mm]
    & q & \mapsto & \eta_{\sQ}(q, \cdot ) \\[4mm]
    \delta_{\sT_X} : & \sT_{X^\circ} & \to & \Omega^{n-1}_{X^\circ} \\[1mm]
    & \vec v & \mapsto & \bigl(\eta\wedge\mu \bigr)(\vec v, \cdot ).
  \end{array}
\end{equation}

\begin{rem}
  If $r = 1$, then $\wedge^{r-1} \sQ^*|_{X^\circ} = \sO_{X^\circ}$ is simply the
  sheaf of functions.
\end{rem}

Using the isomorphisms \eqref{eq:7so} and the complementary form $\mu$, we can
now define a sheaf morphism $\phi: \sQ|_{X^\circ} \to \sT_{X^\circ}$ as
the composite of the following natural maps
\begin{equation}\label{eq:phi}
  \sQ|_{X^\circ} \xrightarrow{\;\;\delta_{\sQ}\;\;} \wedge^{r-1}
  \sQ^*|_{X^\circ} \xrightarrow{\;\;\wedge^{r-1}\beta^*\;\;}
  \Omega^{r-1}_{X^\circ} \xrightarrow{\;\; \wedge \mu \;\;} \Omega^{n-1}_{X^\circ}
  \xrightarrow{\;\;\delta_{\sT_X}^{-1}\;\;} \sT_{X^\circ}.
\end{equation}

\begin{rem}
  In case where $r = 1$, the sheaves $\wedge^{r-1} \sQ^*|_{X^\circ}$ and
  $\Omega^{r-1}_{X^\circ}$ both equal the trivial sheaf $\sO_{X^\circ}$. The
  morphism $\wedge^{r-1}\beta^*$ is then the identity map.
\end{rem}

To end the proof of Corollary~\ref{cor:split}, it will now suffice to prove the
following claim.

\begin{claim}\label{claim:XXL}
  The morphism $\phi: \sQ|_{X^\circ} \to \sT_{X^\circ}$ defines a splitting of
  Sequence~\eqref{eq:quot} over the open set $X^\circ$.
\end{claim}

\subsubsection*{Step 3 in the proof of Corollary~\ref*{cor:split}: preparation for proof of Claim~\ref*{claim:XXL}}

It suffices to show Claim~\ref{claim:XXL} locally, over sufficiently small open
sets $U \subseteq X^\circ$. We will prove Claim~\ref{claim:XXL} by explicit
computation, choosing frames for the bundles $\sE$, $\sT_X$ and $\sQ$ to write
down the morphism $\phi$ and all relevant differential forms. Indeed, choosing
$U$ small enough, we can find frames
\begin{align*}
  & \; e_1, \ldots, e_{n-r}, && \text{\ldots frame of $\sE|_U$}, \\
  & \; \vec q_1, \ldots, \vec q_r, \alpha(e_1), \ldots, \alpha(e_{n-r}) && \text{\ldots frame of $\sT_X|_U$, and}\\
  & \; \beta(\vec q_1), \ldots, \beta(\vec q_{n-r}) && \text{\ldots frame of $\sQ|_U$}.
  \intertext{To simplify notation, set
    $$
    \vec e_i := \alpha(e_i) \in \sT_X(U) \text{\quad and \quad} q_j := \beta(\vec
    q_j) \in \sQ(U).
    $$
    We denote the dual frames by
  }
  & e_1^*, \ldots, e_{n-r}^* && \text{\ldots frame of $\sE^*|_U$}, \\
  & \vec q_1^* \ldots, \vec q_r^*, \vec e_1^*, \ldots, \vec e_{n-r}^* && \text{\ldots frame of $\Omega^1_X|_U$, and} \\
  & q_1^*, \ldots, q_r^* && \text{\ldots frame of $\sQ^*|_U$}.
\end{align*}
Observe that $\alpha^* (\vec e_i^*) = e_i^*$ and $\beta^*(q_j^*) = \vec q_j^*$,
for all indices $i$ and $j$. Scaling the frame $\vec q_1, \ldots, \vec q_r$
appropriately, we may assume that
\begin{equation}\label{eq:balthasar}
  \eta_{\sQ}|_U = q_1^* \wedge \cdots \wedge q_r^* \quad\text{and}\quad 
  \eta|_U = \vec q_1^* \wedge \cdots \wedge \vec q_r^*.
\end{equation}
Scaling the frame $e_1, \ldots, e_{n-r}$, we can then find forms $\sigma_1,
\ldots \sigma_r \in \Omega_X^{n-r-1}(U)$ such that the complementary form $\mu$
can be written as
\begin{equation}\label{eq:caspar}
  \mu|_U = \vec e_1^* \wedge \cdots \wedge \vec e_{n-r}^* + \sum_{i=1}^r \vec q_i^*
  \wedge \sigma_i.
\end{equation}
\begin{rem}
  We do not claim that Equation~\eqref{eq:caspar} defines the forms $\sigma_i$
  uniquely. In fact, there will almost always be several ways to write $\mu|_C$
  in this way. If $n-r-1 = 0$, then the $\sigma_i$ are just functions.
\end{rem}
\begin{rem}\label{rem:ewm}
  On the open set $U$, Equations~\eqref{eq:balthasar} and \eqref{eq:caspar}
  together imply that the globally defined form $\eta \wedge \mu$ is given as
  $$
  \bigl( \eta \wedge \mu \bigr)|_U = \vec q_1^* \wedge \cdots \wedge \vec q_r^*
  \wedge \vec e_1^* \wedge \cdots \wedge \vec e_{n-r}^*.
  $$
\end{rem}

If $n-r-1 > 0$, then the forms $\sigma_i$ of Equation~\eqref{eq:caspar} can be
decomposed further, writing them as sums of pure tensors that only involve $\vec
e_1^*, \ldots, \vec e_{n-r}^*$, and tensors that involve $\vec q_1^*, \ldots,
\vec q_r^*$,
\begin{equation}\label{eq:melchior}
  \sigma_i = \sum_{j=1}^{n-r} a_{ij} \cdot \vec e_1^* \wedge \cdots \xcancel{\vec e_{j}^*} \cdots \wedge \vec e_{n-r}^* + 
  \sum_{k=1}^r \vec q_k^* \wedge \tau_{ik},
\end{equation}
for suitable functions $a_{ij} \in \sO_{X^\circ}(U)$ and forms $\tau_{ik} \in
\Omega_X^{n-r-2}(U)$.

\begin{rem}
  Again, we do not claim that Equation~\eqref{eq:melchior} defines the forms
  $\tau_{ik}$ uniquely. In contrast, note that the functions $a_{ij}$
  are uniquely determined by \eqref{eq:caspar} and \eqref{eq:melchior}.
\end{rem}

\subsubsection*{Step 4 in the proof of Corollary~\ref*{cor:split}: proof of Claim~\ref*{claim:XXL} and end of proof}

We will prove Claim~\ref*{claim:XXL} only in case where $n-r-1 > 0$. The case
where $n-r=1$ follows exactly the same pattern, but is easier. To be precise, we
will prove that
\begin{equation}\label{eq:aD}
  \begin{array}{rccccrccc}
    \phi|_U : & \sQ|_{U} & \to & \sT_X|_{U} \\[1mm]
    &  q_{\ell} & \mapsto & \vec q_{\ell} + \sum_{j=1}^{n-r} \pm a_{\ell j} \cdot \vec e_j
  \end{array}
\end{equation}
where the $a_{\ell j} \in \sO_{X^\circ}(U)$ are the functions introduced in
Equation~\eqref{eq:melchior} above. 
Therefore, $\beta|_U \circ \phi|_U = \mathrm {id}_{\sQ|_U},$ establishing Claim~\ref{claim:XXL}.
By definition of $\phi$, \eqref{eq:aD} is equivalent
to showing that
\begin{equation}\label{eq:aD2}
  \delta_{\sT_X}^{-1} \left( \left( (\wedge^{r-1} \beta^*) \left(\delta_{\sQ}(q_{\ell}) \right) 
    \right) \wedge \mu \right) = \vec q_{\ell} + \sum_{j=1}^{n-r} \pm a_{\ell j} \cdot \vec e_j \quad \text{for all indices $\ell$.}
\end{equation}

The computation proving \eqref{eq:aD2} uses the following elementary
observation.

\begin{obs}\label{obs:dcoo}
  It follows immediately from \eqref{eq:balthasar} and from Remark~\ref{rem:ewm}
  that the sheaf morphisms $\delta_{\sQ}$ and $\delta_{\sT_X}$ introduced in
  \eqref{eq:7so} have the following explicit description on $U$,
  $$
  \begin{array}{rccccrccc}
    \delta_{\sQ}|_U : & \sQ|_{U} & \to & \wedge^{r-1} \sQ^*|_{U} \\[1mm]
    &  q_{\ell} & \mapsto & (-1)^{\ell+1} \cdot q_1^* \wedge \cdots \xcancel{q_{\ell}^*} \cdots \wedge q_r^* \\[4mm]
    \delta_{\sT_X}|_U : & \sT_X|_{U} & \to & \Omega^{n-1}_X|_{U} \\[1mm]
    & \vec q_{\ell} & \mapsto & (-1)^{\ell+1\phantom{r+}} \cdot \vec q_1^* \wedge \cdots \xcancel{\vec q_{\ell}^*} \cdots \wedge \vec q_r^* \wedge  \vec e_1^* \wedge \cdots \wedge \vec e_{n-r}^*\\
    & \vec e_{\ell} & \mapsto & (-1)^{r+\ell+1} \cdot \vec q_1^* \wedge \cdots \wedge \vec q_r^* \wedge  \vec e_1^* \wedge \cdots \xcancel{e_{\ell}^*} \cdots \wedge \vec e_{n-r}^*.
  \end{array}
  $$
\end{obs}

With Observation~\ref{obs:dcoo} in place, Equation~\eqref{eq:aD2} is now shown
easily by direct computation as follows. 
\begin{align*}
  {\sf A} & := \delta_{\sQ}(q_{\ell}) \\
  & \phantom{:}= (-1)^{\ell+1} \cdot q_1^* \wedge \cdots \xcancel{q_{\ell}^*} \cdots \wedge q_r^* && \text{Obs.~\ref{obs:dcoo}} \displaybreak[0]\\[4mm]
  {\sf B} & := \bigl( \wedge^{r-1}\beta^*\bigr)({\sf A}) \\
  & \phantom{:}= (-1)^{\ell+1} \cdot \vec q_1^* \wedge \cdots \xcancel{q_{\ell}^*} \cdots \wedge \vec q_r^* && \text{Defn. of $\beta^*$} \displaybreak[0]\\[4mm]
  {\sf C} & := {\sf B} \wedge \mu \\
  & \phantom{:}= (-1)^{\ell+1} \cdot \vec q_1^* \wedge \cdots \xcancel{q_{\ell}^*} \cdots \wedge \vec q_r^* \wedge \left( \vec e_1^* \wedge \cdots \wedge \vec e_{n-r}^* + \sum_{i=1}^r \vec q_i^* \wedge \sigma_i \right) && \text{by \eqref{eq:caspar}}\\
  & \phantom{:}= (-1)^{\ell+1} \cdot \vec q_1^* \wedge \cdots \xcancel{q_{\ell}^*} \cdots \wedge \vec q_r^* \wedge \left( \vec e_1^* \wedge \cdots \wedge \vec e_{n-r}^* + \vec q_{\ell}^* \wedge \sigma_i \right) \\
  & \phantom{:}= (-1)^{\ell+1} \cdot  \vec q_1^* \wedge \cdots \xcancel{q_{\ell}^*} \cdots \wedge \vec q_r^* \wedge \Biggl( \vec e_1^* \wedge \cdots \wedge \vec e_{n-r}^*  \\
  & \qquad\qquad + \vec q_{\ell}^* \wedge \Biggl( \sum_{j=1}^{n-r} a_{\ell j} \cdot \vec e_1^* \wedge \cdots \xcancel{\vec e_{j}^*} \cdots \wedge \vec e_{n-r}^* + \sum_{k=1}^r \vec q_k^* \wedge \tau_{ik} \Biggr) \Biggr) && \text{by \eqref{eq:melchior}}\\
  & \phantom{:}= (-1)^{\ell+1} \cdot  \vec q_1^* \wedge \cdots \xcancel{q_{\ell}^*} \cdots \wedge \vec q_r^* \wedge \vec e_1^* \wedge \cdots \wedge \vec e_{n-r}^* \\
  & \qquad\qquad + \sum_{j=1}^{n-r} \pm a_{\ell j} \cdot  \vec q_1^* \wedge \cdots \wedge \vec q_r^* \wedge \vec e_1^* \wedge \cdots \xcancel{\vec e_{j}^*} \cdots \wedge \vec e_{n-r}^* \displaybreak[0]\\[4mm]
\intertext{and finally}
  {\sf D} & := \delta_{\sT_X}^{-1}({\sf C}) = \vec q_{\ell} + \sum_{j=1}^{n-r} \pm a_{\ell j} \cdot \vec e_j &&  \text{Obs.~\ref{obs:dcoo}}
\end{align*}

This finishes the proof of Equations~\eqref{eq:aD2}, \eqref{eq:aD},
Claim~\ref{claim:XXL}, and hence of Corollary~\ref{cor:split}. \qed

\section{Proof of Theorem~\ref*{decoII}}
\label{sect:deco}

We have divided the proof of Theorem~\ref{decoII} into a sequence of steps, each
formulated as a separate result. Some of these statements might be of
independent interest. The proof of Theorem~\ref{decoII} follows quickly from
these preliminary steps and is given in Section~\ref{sect:proofofdecoII}.

\begin{thm}[Splitting the tangent sheaf of varieties with trivial canonical bundle]\label{deco}
  Let $X$ be a normal $n$-dimensional projective variety with at worst canonical
  singularities. Assume that $\omega_X \cong \sO_X$ and that $\wtilde q(X) = 0$.
  Let $h = (H_1, \ldots, H_{n-1})$ be ample divisors on $X$ and assume that
  $\sT_X$ is not $h$-stable. Let $0 \subsetneq \sE \subsetneq \sT_X$ be a
  saturated destabilising subsheaf, that is, a proper subsheaf with non-negative
  slope $\mu_h(\sE) \geq 0$ and torsion free quotient $\sT_X / \sE$.

  Then there exists a number $M \in \bN^+$ such that $(\det \sE)^{[M]} \cong
  \sO_X$. Further, there exists a finite cover $f: \wtilde X \to X$, étale in
  codimension one, and a proper subsheaf $\sF \subsetneq \sT_{\wtilde X}$ such
  that the following holds.
  \begin{enumerate}
  \item\label{il:split} The tangent sheaf of $\wtilde X$ decomposes as a direct
    sum, $\sT_{\wtilde X} \cong \bigl( f^{[*]}\sE \bigr) \oplus \sF$.
  \item\label{il:trivdet} Both summands in (\ref{deco}.\ref{il:split}) have
    trivial determinant. In other words, $\det f^{[*]}\sE \cong \sO_{\wtilde X}$
    and $\det \sF \cong \sO_{\wtilde X}$.
  \end{enumerate}
\end{thm}

Before proving Theorem~\ref{deco} in Section~\ref{ssec:profdeco} below, we note
an important corollary, obtained from Theorem~\ref{deco} by iterated
application. The following notation, which summarises the conditions spelled out
in Condition~(\ref{decoII}.\ref{il:B}) of the Decomposition
Theorem~\ref{decoII}, is used in its formulation.

\begin{defn}[Strong stability]\label{def:strongStab}
  Let $X$ be a normal projective variety of dimension $n$, and $\sF$ a reflexive
  coherent sheaf of $\sO_X$-modules. We call $\sF$ \emph{strongly stable}, if
  for any finite morphism $f : \wtilde X \to X$ that is étale in codimension
  one, and for any choice of ample divisors $\wtilde H_1, \ldots, \wtilde
  H_{n-1}$ on $\wtilde X$, the reflexive pull-back $f^{[*]} \sF$ is stable with
  respect to $(\wtilde H_1, \ldots, \wtilde H_{n-1})$.
\end{defn}

\begin{cor}[Existence of a decomposition]\label{cor:deco}
  Let $X$ be a normal projective variety having at worst canonical
  singularities. Assume that $\omega_X \cong \sO_X$. Then there exists a finite
  cover $f: \wtilde X \to X$, étale in codimension one, and a decomposition
  $$
  \sT_{\wtilde X} \cong \bigoplus \sE_i,
  $$
  where the $\sE_i$ are strongly stable subsheaves of $\sT_{\wtilde X}$ with
  trivial determinants, $\det \sE_i \cong \sO_{\wtilde X}$.
\end{cor}

\begin{rem}
  We note that the summands $\sE_i$ in the decomposition established in
  Corollary~\ref{cor:deco} are automatically saturated. Indeed, as a subsheaf of
  the torsion-free sheaf $\sT_{\widetilde X}$ each $\sE_i$ is torsion-free. The
  quotient of $\sT_{\widetilde X}$ by any of the summands is a direct sum of the
  remaining summands, hence torsion-free.
\end{rem}

\begin{proof}[Proof of Corollary~\ref{cor:deco}]
  We need to find a cover $f: \wtilde X \to X$ and a decomposition $\sT_{\wtilde
    X} \cong\bigoplus \sE_i$, such that all factors $\sE_i$ are strongly
  stable. Since the rank of $\sT_X$ is finite, there exists a finite cover $f:
  \wtilde X \to X$, étale in codimension one, with a proper decomposition
  \begin{equation}\label{eq:stablecomp}
    \sT_{\wtilde X} \cong\bigoplus_{i\geq 1} \sE_i,
  \end{equation}
  in which the number of direct summands is maximal.  We claim that each summand
  is then automatically strongly stable.

  We argue by contradiction and assume that there exists a further finite cover,
  $g: \what X \to \wtilde X$, étale in codimension one, and a list of ample
  divisors $\what h = (\what H_1, \ldots, \what H_{n-1})$ on $\what X$ such that
  the reflexive pull-back of one of the summands, say $\what \sE_1 :=
  g^{[*]}\sE_1$, is not stable with respect to $\what h$. Let $0 \subsetneq
  \what \sS \subsetneq \what \sE_1$ be a $\what h$-destabilising subsheaf. By
  \cite[Prop.~7.6(b)]{Kob87}, whose proof carries over without change from the
  smooth to the singular setup, we may assume that $\widehat \sS$ is saturated
  in $\what \sE_1$ and therefore also in $\sT_{\what X}$. Since $\what \sE_1$
  and $\sT_{\what X}$ both have vanishing $\what h$-slope, the sheaf $\what \sS$
  is also a destabilising subsheaf for $\sT_{\what X}$. Replacing $\what X$ by a
  further cover, if necessary, Theorem~\ref{deco} therefore allows to assume
  without loss of generality that the tangent sheaf splits, say $\sT_{\what X} =
  \what \sS \oplus \what \sQ$. Since the sheaves $\sT_{\what X}$, $\what \sE_1$,
  $\what \sS$ and $\what \sQ$ are all locally free on the smooth locus of $\what
  X$, elementary linear algebra gives a decomposition
  $$
  \what \sE_1|_{\what X_{\reg}} = \what \sS|_{\what X_{\reg}} \,\, \oplus \,\, (\what
  \sE_1 \cap \what \sQ)|_{\what X_{\reg}}.
  $$
  Taking double duals, we obtain a decomposition $\what \sE_1 = \what \sS^{**}
  \oplus (\what \sE_1 \cap \what \sQ)^{**}$, which contradicts maximality of the
  decomposition~\eqref{eq:stablecomp} and therefore finishes the proof of
  Corollary~\ref{cor:deco}.
\end{proof}

\begin{rem}[Uniqueness of the decomposition]\label{rem:uniqueness}
  Given a variety $X$ as in Corollary~\ref{cor:deco}, let $f_1 : \wtilde X_1 \to
  X$ and $f_2 : \wtilde X_1 \to X$ be two finite morphisms, étale in codimension
  one, such that the tangent bundles split into strongly stable summands,
  $$
  \sT_{\wtilde X_1} \cong \bigoplus_{i=1}^N \sE^1_i \quad\text{and}\quad \sT_{\wtilde
    X_2} \cong \bigoplus_{j=1}^M \sE^2_j.
  $$
  Let $\what X$ be an irreducible component of the normalisation of the fibered
  product $\wtilde X_1 \times_X \wtilde X_2$. We obtain a
  diagram
  $$
  \xymatrix{ %
    \what X \ar[r]^{g_1} \ar[d]_{g_2} & \wtilde X_1 \ar[d]^{f_1} \\
    \wtilde X_2 \ar[r]_{f_2} & X,
  }
  $$
  where $g_1$, $g_2$ are again finite and étale in codimension one. Since
  $\sT_{\wtilde X} \cong g_1^{[*]} \sT_{\wtilde X_1} \cong g_2^{[*]}
  \sT_{\wtilde X_2}$, we obtain decompositions
  \begin{equation}\label{eq:DwhatTX}
    \sT_{\what X} \,\, \cong \,\, \bigoplus_{i=1}^{N} g_1^{[*]}\sE^1_i \,\, \cong \,\, \bigoplus_{j=1}^M g_2^{[*]}\sE^2_j.
  \end{equation}
  Choosing any ample polarisation on $\what X$, stability of the summands
  implies that any morphism $\what \sE^1_i \to \what \sE^2_j$ must either be
  trivial, or an isomorphism. It follows that the
  decompositions~\eqref{eq:DwhatTX} satisfy the following extra conditions
  \begin{enumerate-p}
    \setcounter{enumi}{\value{equation}}
  \item the number of summands in the decompositions agrees, $N=M$, and

  \item up to permutation of the summands we have isomorphisms $g_1^{[*]}\sE^1_i
    \cong g_2^{[*]}\sE^2_i$ for all $i \in \{1, \dots, N\}$.
  \end{enumerate-p}
  In that sense, the decomposition found in Corollary~\ref{cor:deco} is unique.
\end{rem}

\subsection{Proof of Theorem~\ref*{deco}}
\label{ssec:profdeco}

Since the proof of Theorem~\ref{deco} is somewhat long, we have subdivided it
into four relatively independent steps, given in
Sections~\ref{sssec:S1}--\ref{sssec:S4} below. Figure~\vref{fig:ssx} gives an
overview of the spaces and morphisms constructed in the course of the proof.

\begin{figure}
  \centering
  
  $$
  \xymatrix{ %
    \wtilde X' \ar[rrrrr]^{g,\text{ index-one cover for }\sE'}_{\text{finite,
        étale in codimension one}} \ar@{=}[d] &&&&& X' \ar[rrrrr]^{\phi, \text{
        $\bQ$-factorialisation}}_{\text{small, birational}} &&&&&
    X  \ar@{=}[d] \\
    \wtilde X' \ar[rrrrr]^{\psi,\text{ Stein factorisation I}}_{\text{connected
        fibers, small, birational}} &&&&& \wtilde X \ar[rrrrr]^{f,\text{ Stein
        factorisation II}}_{\text{finite, étale in codimension one}} &&&&& X }
  $$

  \caption{Spaces and morphisms constructed in the proof of Theorem~\ref*{deco}}
  \label{fig:ssx}
\end{figure}

\subsubsection{Step~1: The subsheaf $\sE \subsetneq \sT_X$}
\label{sssec:S1}

To start the proof of Theorem~\ref{deco}, we discuss the structure of the
saturated destabilising sheaf $\sE \subsetneq \sT_X$.  First note that due to
torsion-freeness of $\sE$ and of $\sT_X/\sE$, the sheaf $\sE$ is a
sub-vectorbundle of $\sT_X$ outside of a set of codimension two. Next, we
compute its slope.

\begin{lem}[Slopes of destabilising subsheaves]\label{lem:slopeZero}
  In the setup of Theorem~\ref{deco}, any destabilising subsheaf of $\sT_X$ has
  slope zero. In particular, $\mu_h(\sE)=0$.
\end{lem}
\begin{proof}
  Let $\sG$ be any destabilising subsheaf of $\sT_X$. Since $K_X$ is assumed to
  be trivial, it follows that $\sG$ has non-negative slope, $\mu_h(\sG) \geq 0$.
  On the other hand, we know from Proposition~\ref{prop:semistable} that $\sT_X$
  is $h$-semistable. Consequently, we have $\mu_h(\sG) = 0$, as claimed.
\end{proof}

\subsubsection{Step~2: The $\mathbb{Q}$-factorialisation of $X$}
\label{sssec:S2}

Let $\phi: X' \to X$ be a $\bQ$-factorialisation of $X$, that is, a small
birational morphism where $X'$ is $\bQ$-factorial and has only canonical
singularities. The existence of $\phi$ is established in
\cite[Lem.~10.2]{BCHM06}. Since $\phi$ is small, it is clear that $\omega_{X'}
\cong \phi^* \omega_X \cong \sO_{X'}$, and that $\wtilde q(X') = 0$.  Set $\sE'
:= \phi^{[*]}(\sE)$. Since $\sE'$ injects into the tangent sheaf $\sT_{X'}$ away
from a set of codimension two, it follows from reflexivity that $\sE'$ injects
into $\sT_{X'}$ everywhere. We can therefore view it as a proper subsheaf $\sE'
\subsetneq \sT_{X'}$. Notice that $\sE'$ is saturated in $ \sT_{X'}$, since
$\phi$ is small.

\begin{claim}\label{claim:coverSuffices}
  To prove statements (\ref{deco}.\ref{il:split}) and
  (\ref{deco}.\ref{il:trivdet}) of Theorem~\ref{deco}, it suffices to find a
  finite cover $g: \wtilde X' \to X'$, étale in codimension one, and a
  decomposition
  \begin{equation}\label{eq:rsdx}
    \sT_{\wtilde X'} \cong \sF' \oplus \left(g^{[*]}\sE' \right)
  \end{equation}
  where $\det \sF' \cong \det \left(g^{[*]}\sE'\right) \cong \sO_{\wtilde X'}$.
\end{claim}
\begin{proof}
  As indicated in Figure~\ref{fig:ssx}, consider the Stein factorisation of the
  composed morphism $\phi \circ g$,
  $$
  \xymatrix{%
    \wtilde X' \ar[rrr]_{\psi, \text{ small birational}}
    \ar@/^0.4cm/[rrrrrrr]^{\phi \circ g} &&& 
    \wtilde X \ar[rrrr]_{f, \text{ finite, étale in codim.~1}} &&&& 
    X.
  }
  $$
  The reflexive push-forward of~\eqref{eq:rsdx},
  $$
  \sT_{\wtilde X} \cong \bigl(\psi_* \sT_{\wtilde X'} \bigr)^{**} \cong
  \left(\psi_* \bigl( \sF' \oplus g^{[*]}\sE' \bigr)\right)^{**} \cong
  \underbrace{\bigl( \psi_*\sF' \bigr)^{**}}_{=: \sF} \oplus
  \underbrace{\bigl(\psi_* (g^{[*]}\sE') \bigr)^{**}}_{\cong f^{[*]}\sE},
  $$
  then yields a decomposition on $\wtilde X$ that satisfies both
  (\ref{deco}.\ref{il:split}) and (\ref{deco}.\ref{il:trivdet}), thus finishing
  the proof of Claim~\ref{claim:coverSuffices}.
\end{proof}

\subsubsection{Step~3: High reflexive powers of $\det \sE$ and  $\det \sE'$ are trivial}
\label{sssec:S3}

We need to show that high reflexive powers of $\det \sE$ and $\det \sE'$ are
trivial. To this end, recall that the reflexive sheaf $\det \sE'$ is a Weil
divisorial sheaf on $X'$. In other words, there exists a Weil divisor $D'$ on
$X'$ so that $\det \sE' \cong \sO_{X'}(D')$. Since $X'$ is $\bQ$-factorial,
there exists be a number $m \in \bN^+$ such that $mD'$ is actually Cartier.

As a first step towards showing triviality of a sufficiently high multiple, we
prove numerical triviality of $D'$.

\begin{lem}
  Setting as above, then the $\bQ$-Cartier divisor $D'$ is numerically trivial.
\end{lem}
\begin{proof}
  We aim to apply Proposition~\ref{prop:angela} in order to conclude that $D'$
  is numerically trivial. As a first step in this direction, recall from
  Lemma~\ref{lem:slopeZero} that $\mu_h(\sE) \cdot h = 0$. Consequently, we have
  the following equality of intersection numbers of $\bQ$-Cartier divisors,
  \begin{equation}\label{eq:intersectH}
    0= D' \cdot \phi^*(H_1) \cdots \phi^*(H_{n-1}) = -D' \cdot \phi^*(H_1) \cdots \phi^*(H_{n-1}).
  \end{equation}
  Secondly, observe that the inclusion $\sE' \into \sT_{X'}$ yields an inclusion
  $\det \sE' \into \wedge^{[p]} \sT_{X'}$.  Dualising, we obtain a non-zero morphism
  $$
  \Omega^{[p]}_{X'} \to (\det \sE')^* \cong \sO_{X'}(-D').
  $$
  Since $K_X$ is trivial and since $X$ has only canonical singularities, $X$ is
  not uniruled and Proposition~\ref{prop:pseudoeffectivitygeneralised} therefore
  implies that the $\bQ$-Cartier divisor $-D'$ is pseudoeffective on $X'$. Due
  to Equation~\eqref{eq:intersectH}, Proposition~\ref{prop:angela} applies to
  show that $-D'$ and hence $D'$ is numerically trivial indeed.
\end{proof}

Next up, we conclude from numerical triviality that high reflexive powers of
$\det \sE$ and $\det \sE'$ are trivial.

\begin{cor}
  Setting as above. Then there exists a number $M \in \bN^+$ such that $(\det
  \sE')^{[M]} \cong \sO_{X'}$ and $(\det \sE)^{[M]} \cong \sO_X$.
\end{cor}
\begin{proof}
  We are going to use the assumption that $\wtilde q(X) = 0$, which implies that
  $$
  0 = q(X') = h^1\bigl( X',\, \sO_{X'} \bigr),
  $$
  hence the Picard group of $X'$ is discrete. The subgroup $\Pic^0(X')
  \subsetneq \Pic(X')$ of invertible sheaves with numerically trivial Chern
  class is therefore finite. It follows that there exists a positive natural
  number $k$ such that $(\det \sE')^{[km]} \cong \sO_{X'}$. Since the reflexive
  sheaves $\sO_X \cong \phi_* \bigl( (\det \sE')^{[km]} \bigr)$ and $(\det
  \sE)^{[km]}$ agree in codimension one, they agree everywhere, and $(\det
  \sE)^{[km]}$ is likewise trivial.
\end{proof}

\subsubsection{Step~4: Constructing the splitting on a cover of $X'$}
\label{sssec:S4}

Since $\det \sE'$ is torsion, there exists an index-one cover $g: \wtilde X' \to
X'$ for $\det \sE'$, see for example~\cite[2.52]{KM98}. This is a finite
morphism from a normal variety, étale in codimension one, such that
\begin{equation}\label{eq:detf*trivial}
  g^{[*]} \, \det \sE' \cong \sO_{\wtilde X'}.
\end{equation}
Recall from \cite[5.20]{KM98} that $\wtilde X'$ has trivial canonical bundle and
at worst canonical singularities. Set $\wtilde \sE' := g^{[*]}\sE'$.  We aim to
construct a splitting of $\sT_{\wtilde X'}$ using the existence of complementary
subsheaves shown in Corollary~\ref{cor:split} on page~\pageref{cor:split}. The
following claim guarantees that the assumptions made in
Corollary~\ref{cor:split} are satisfied in our context.

\begin{claim}\label{claim:XXZ}
  The inclusion $\sE' \into \sT_{X'}$ induces an inclusion $\wtilde \sE' \into
  \sT_{\wtilde X'}$. With the inclusion understood, $\wtilde \sE'$ is a
  saturated subsheaf of $\sT_{\wtilde X'}$
\end{claim}
\begin{proof}
  Since $\phi \circ g: \wtilde X' \to X$ is étale in codimension one, the
  reflexive sheaf $\wtilde \sE'$ injects into $\sT_{\wtilde X'}$ outside of a
  small set, and is a saturated subsheaf there. Since both $\wtilde \sE'$ and
  its saturation in $\sT_{\wtilde X'}$ are reflexive and agree in codimension
  one, it follows that $\wtilde \sE'$ actually is isomorphic to its saturation,
  as claimed.
\end{proof}

With Claim~\ref{claim:XXZ} in place, Corollary~\ref{cor:split} asserts the
existence of a sheaf $\wtilde \sF'$ with trivial determinant such that
$\sT_{\wtilde X'} \cong \wtilde \sE' \oplus \wtilde \sF'$. As we have seen in
Claim~\ref{claim:coverSuffices}, this concludes the proof of
Theorem~\ref{deco}. \qed

\subsection{Integrability of direct summands}

In this section, we show that the individual summands in the decomposition
stated in Corollary~\ref{cor:deco} are integrable; that is, they define
foliations.

\begin{thm}[Integrability of direct summands]\label{thm:integrable}
  Let $X$ be a normal projective variety with at worst canonical
  singularities. Assume that $\omega_X \cong \sO_X$. Let $\sT_X \cong \bigoplus
  \sE_i$ be a decomposition into reflexive sheaves with trivial determinants.
  Then all $\sE_i$ are integrable.
\end{thm}
\begin{proof}
  We follow the arguments of \cite{Hoe07}. Without loss of generality we assume
  that $\sT_X \cong\sE_1 \oplus \sE_2$, that is, we assume that $\sT_X$ can be
  decomposed into two summands. We will show that $\sE_2$ is integrable. The
  integrability of $\sE_1$ then follows for symmetry reasons. Let $r_1$ be the
  rank of $\sE_1$, and consider the trivialisable sheaf $\sL_1 := \det \sE_1$.
  Since $\sE_1^*$ is a direct summand of $\Omega_X^1$, the reflexive sheaf
  $\sL_1 \otimes \Omega_X^{[r_1]}$ has a trivial direct summand. Let $\theta \in
  H^0\bigl(X, \, \sL_1\otimes \Omega_X^{[r_1]}\bigr)$ be the corresponding
  nowhere vanishing $ \sL_1$-valued differential form and let $\pi: \wtilde X
  \to X$ be a resolution of singularities. By the Extension
  Theorem~\ref{thm:ext}, the reflexive differential form $\theta$ pulls back to
  a non-trivial section $\wtilde \theta \in H^0\bigl(\wtilde X, \, \pi^*(\sL_1)
  \otimes \Omega_{\wtilde X}^{r_1}\bigr)$.

  At general points of $\wtilde X$, where $\pi$ is isomorphic, the sheaf $\pi^*
  \sE_2$ coincides with the degeneracy sheaf $S_{\wtilde \theta}$ of $\wtilde
  \theta$, that is, the sheaf of vector fields $\vec v$ such that the
  contraction
  $$
  i_{\vec v}(\wtilde \theta) = \wtilde \theta ( \vec v, \cdot ) \in H^0 \bigl(
  \wtilde X,\, \pi^*(\sL_1) \otimes \Omega^{r_1-1}_{\wtilde X})
  $$
  vanishes\footnote{Degneracy subsheaves are introduced and discussed in more
    detail in Section~\ref{subsect:nondeg} below.}. In this setting, it follows
  from \cite[Main~Thm.]{Dem02} that $S_{\wtilde \theta}$ is integrable. As a
  consequence, we obtain that $\sE_2$ is integrable at general points of
  $X$. Since $\sE_2$ is a saturated subsheaf of $\sT_X$, it follows that it is
  integrable everywhere.
\end{proof}

\subsection{Proof of Theorem~\ref*{decoII}}\label{sect:proofofdecoII}

We maintain notation and assumptions of Theorem~\ref{decoII}.
Corollary~\ref{cor:qgleichnull} implies the existence of an Abelian variety $A$
and of a projective variety $X'$ with at worst canonical singularities, with
trivial canonical bundle and $\wtilde q (X') = 0$, together with a finite cover
$A \times X' \to X$, étale in codimension one. Property
(\ref{decoII}.\ref{il:C}) stated in Theorem~\ref{decoII} is hence fulfilled for
any cover of the form $A \times \widetilde X \to A \times X'$, where $\wtilde X
\to X'$ is a finite cover, étale in codimension one. The existence of such a
cover $\wtilde X \to X$ and of a decomposition of $\sT_{\widetilde X}$
satisfying Properties (\ref{decoII}.\ref{il:A}) and (\ref{decoII}.\ref{il:B})
follows by combining Corollary~\ref{cor:deco} and
Theorem~\ref{thm:integrable}. In summary, this finishes the proof of
Theorem~\ref{decoII}. \qed

\begin{rem}
  The decomposition theorem of Beauville-Bogomolov holds for compact Kähler
  manifolds.  Therefore, we should expect a singular version in the
  non-algebraic context as well.  In particular, Theorem~\ref{decoII} should
  hold for Kähler varieties. There are however two main ingredients in our
  argument which are not yet available in the Kähler context: the Extension
  Theorem~\ref{thm:ext} and the pseudoeffectivity result
  Proposition~\ref{prop:pseudoeffectivitygeneralised}.
\end{rem}

\section{Towards a structure theory}
\label{subsect:fromtangentbundletoX}

If $X$ is any projective manifold with Kodaira dimension zero, $\kappa(X)=0$,
standard conjectures of minimal model theory predict the existence of a
birational contraction\footnote{Following standard use, we call a birational map
  a \emph{contraction map} if its inverse does not contract any divisors.} map
$\lambda : X \dasharrow X_\lambda$, where $X_\lambda$ has terminal singularities
and numerically trivial canonical divisor. Generalising the Beauville-Bogomolov
Decomposition Theorem~\ref{bb}, it is widely expected that $X_\lambda$ admits a
finite cover, étale in codimension one, which can be birationally decomposed
into a product
$$
T \times \prod X_j, 
$$
where $T$ is a torus and the $X_j$ are singular versions of Calabi-Yau manifolds
and irreducible symplectic manifold, which cannot be decomposed further. Such a
decomposition result would clearly be a central pillar to any structure theory
for varieties with Kodaira dimension zero. The main result of the present paper,
Theorem~\ref{decoII}, is a first step in this direction.

Section~\ref{subsect:decomposition} discusses the remaining problems of turning
the decomposition found in the tangent sheaf into a decomposition of the
variety. Section~\ref{subsect:CY} gives a conjectural description of the
irreducible pieces coming out of the decomposition, discussing singular
analogues of Calabi-Yau and irreducible holomorphic-symplectic varieties, and
proving the conjectured description in low dimensions.  Finally, fundamental
groups of varieties with trivial canonical class, which are crucial for our
understanding of this class of varieties, are discussed in the concluding
Section~\ref{subsect:fundamentalgroups}.

\begin{rem}
  Corollary~\ref{cor:qgleichnull} and Theorem~\ref{decoII} allow to restrict our
  attention to varieties with vanishing augmented irregularity. For most of the
  present Section~\ref{subsect:fromtangentbundletoX}, we will therefore only
  consider varieties $X$ with $\wtilde q(X)=0$.
\end{rem}

\subsection{Decomposing varieties with trivial canonical bundle}
\label{subsect:decomposition}

In technically correct terms, the setup of our discussion is now summarised as
follows.

\begin{setup}\label{setup:82}
  Let $X$ be a normal $\bQ$-factorial projective variety with canonical
  singularities such that $K_X$ is torsion and $\wtilde q(X) = 0$. By
  Theorem~\ref{decoII}, there exists a finite cover $f: \wtilde X \to X$, étale
  in codimension one, such that $\omega_X = \sO_X$ and such that there exists a
  decomposition
  $$
  \sT_{\wtilde X} = \bigoplus \sE_i
  $$
  of $\sT_{\wtilde X}$ into strongly stable integrable reflexive subsheaves.
\end{setup}

In view of the desired decomposition of the variety $\wtilde X$, this naturally
leads to the following problems.

\begin{problem}[Algebraicity of leaves]\label{pb:AoL}
  In Setup~\ref{setup:82}, show that the leaves of the foliations $\sE_i$ are
  algebraic, perhaps after passing to another cover.
\end{problem}

\begin{problem}[Decomposition of the variety]\label{pbl} 
  In the setup of Problem~\ref{pb:AoL}, show that the algebraicity of the leaves
  leads to a birational decomposition of $\wtilde X$, perhaps after passing to
  another cover. More precisely, show that there is a birational morphism
  $$
  g : \wtilde X \dasharrow \prod X_i,
  $$
  isomorphic outside of a small set $V \subset \wtilde X$, such that the
  following holds.
  \begin{enumerate}
  \item The varieties $X_j$ are smooth, projective with $\kappa (X_j) = 0$ for
    all $j$.
  \item If $p_j$ denotes the composition of $g$ with $j^{\rm th}$ projection
    $\Pi\, X_i \to X_j$, then $p_j^*(\sT_{Y_j}) = \sE_j$ over $X \setminus V$
    for all $j$.
\end{enumerate}  
\end{problem}

\begin{rem}
  Once it is known that the leaves of $\sE_j$ are algebraic, one easily obtains
  rational maps $X \dasharrow Y_j$ to smooth projective varieties such that
  $\sE_j = \sT_{X/Y_j}$ generically.  The main problem is now to show that the
  equality $\sE_j = \sT_{X/Y_j}$ holds everywhere, and that $\kappa (Y_j) = 0$.
\end{rem} 

A solution to Problem~\ref{pbl} is not the yet desired final outcome of our
decomposition strategy for $X$: since $K_X$ is (numerically) trivial one clearly
aims for a decomposition into varieties with trivial canonical class.  Assuming
that the minimal model program works for varieties of Kodaira dimension zero,
each $X_j$ may be replaced by a minimal model $X'_j$. As a consequence we would
obtain a birational map $g': X \dasharrow \Pi\, X'_i$. If the singularities of
$X'$ are not only canonical but terminal, it follows from \cite{Kawamata08} that
$g'$ is isomorphic in codimension one and decomposes into a finite sequence of
flops. One might hope that each terminal variety with numerically trivial
canonical class decomposes into terminal varieties with trivial canonical class
and strongly stable tangent bundle, after performing a finite cover, étale in
codimension one, and after performing a finite number of flops.

\subsection{Classifying the strongly stable pieces: Calabi-Yau and irreducible holomorphic-symplectic varieties}
\label{subsect:CY}

We start with after a short discussion of the notion of strong stability in
Section~\ref{ssec:ssvs}, showing by way of example that varieties with strongly
stable tangent sheaves are the ``right'' objects when building a structure theory
for spaces of Kodaira dimension zero. The remainder of the present
Section~\ref{subsect:CY} discusses these spaces in detail.

Sections~\ref{subsect:nondeg} and \ref{ssec:aodf} relate stability properties of
the tangent bundle to non-degeneracy of differential forms, and discuss
implications for the exterior algebra of reflexive forms. We apply these results
in the concluding Section~\ref{sect:singularclassification} to show that
singular varieties with strongly stable tangent sheaf are in a very strong sense
natural analogues of Calabi-Yau and irreducible holomorphic-symplectic manifolds,
at least in dimension up to five. There is ample evidence to conjecture that
this description holds in general, for strongly stable varieties of arbitrary
dimension.

\subsubsection{Strong stability versus stability}
\label{ssec:ssvs}

At first sight, it seems tempting to consider varieties with stable tangent
bundle as the building blocks of varieties with semistable tangent sheaf, such
as varieties with trivial canonical bundle. However, the following example shows
that \emph{strong stability} is indeed the correct notion in our setup.

\begin{ex}[A variety with stable, but not strongly stable tangent sheaf]
  Let $Z$ be a projective K3-surface, let $\wtilde X := Z \times Z$ with
  projections $p_1, p_2: \wtilde X \to Z$, and let $\phi \in \Aut_\sO(\wtilde
  X)$ be the automorphism which interchanges the two factors. The quotient $X :=
  \wtilde X / \langle\phi\rangle = \Sym^2(Z)$ is then a projective
  holomorphic-symplectic variety with trivial canonical bundle and rational
  Gorenstein singularities. The quotient map $\pi : \wtilde X \to X$ is finite
  and étale in codimension one. Let $h$ be any ample polarisation on $X$. The
  tangent sheaf $\sT_X$ of $X$ is obviously not strongly stable.
  
  However, we claim that $\sT_X$ is $h$-stable. Indeed, suppose that there
  exists a non-trivial $h$-stable subsheaf $0 \subsetneq \sS \subsetneq \sT_X$
  with slope zero that destabilises $\sT_X$. Then, the reflexive pull-back
  $\wtilde \sS := \pi^{[*]}(\sS)$ is $\pi^*(h)$-polystable, see
  \cite[Lem.~3.2.3]{HL97}, and injects into $\sT_{\wtilde X} = p_1^*(\sT_Z)
  \oplus p_2^*(\sT_Z)$. Since neither of the two sheaves $p_j^*\sT_Z$ is stable
  under the action of $\phi$, clearly $\wtilde \sS$ is not one of these. More is
  true: looking at the maps to the two summands of $\sT_{\wtilde X}$ and using
  that morphisms between stable sheaves with the same slope are either trivial
  or isomorphic, we see that $\sS$ has to be stable of rank two, and isomorphic
  to both $p_1^*\sT_Z$ and $p_2^*\sT_Z$. This is absurd, as restriction to
  $p_1$-fibers shows that the sheaves $p_1^*\sT_Z$ and $p_2^*\sT_Z$ are in fact
  not isomorphic.
\end{ex}

\subsubsection{Non-degeneracy of differential forms}
\label{subsect:nondeg}

If $X$ is a canonical variety with numerically trivial canonical class,
stability of the tangent bundle has strong implications for the geometry of
differential forms $X$. This section is concerned with degeneracy
properties. Conjectural consequences for the structure of the exterior algebra
of forms are discussed in the subsequent Section~\ref{ssec:aodf}.

Non-degeneracy of differential forms will be measured using the following
definition.

\begin{defn}[Contraction of a reflexive form, degeneracy subsheaf]\label{cont}
  Let $X$ be a normal complex variety, let $0 < q \leq \dim X$ be any number and
  $\sigma \in H^0\bigl(X,\, \Omega^{[q]}_X\bigr)$ any reflexive form. The
  \emph{contraction map of $\sigma$} is the unique sheaf morphism
  $$
  i_\sigma: \sT_X \to \Omega^{[q-1]}_X
  $$
  whose restriction to $X_{\reg}$ is given by $\vec u \mapsto \sigma(\vec u,
  \cdot )$. Let $S_{\sigma} := \ker (i_\sigma)$ be the kernel of $i_\sigma$. We
  call $S_{\sigma} \subseteq \sT_X$ the \emph{degeneracy subsheaf} of the
  reflexive form $\sigma$. If $S_{\sigma} = 0$, we say that $\sigma$ is
  \emph{generically non-degenerate}.
\end{defn}

The main result of the present section asserts that in our setup, forms never
degenerate. This can be seen as first evidence for the conjectural
classification of the stable pieces into ``Calabi-Yau'' and
``irreducible holomorphic-symplectic'' which we discuss later in
Section~\ref{sect:singularclassification} below.

\begin{prop}[Non-degeneracy of forms on canonical varieties with stable $\sT_X$]\label{prop:max2}
  Let $X$ be a normal $n$-dimensional projective variety $X$ having at worst
  canonical singularities, $n > 1$. Assume that the canonical divisor $K_X$ is
  numerically trivial, and that the tangent sheaf $\sT_X$ is stable with respect
  to some ample polarisation.  If $\sigma$ is any non-zero reflexive form on
  $X$, then $\sigma$ is generically non-degenerate, in the sense of
  Definition~\ref{cont}.
\end{prop}
\begin{proof}
  We argue by contradiction and assume that there exists a reflexive $q$-form
  $\sigma$ whose degeneracy subsheaf does not vanish, $S_{\sigma} \ne
  0$. Consider the exact sequence
  \begin{equation}\label{eq:exactSomega}
    0 \to S_{\sigma} \to \sT_X \xrightarrow{i(\sigma)} \underbrace{\Image
      i(\sigma)}_{=: \sE \,\, \subseteq \,\, \Omega_X^{[q-1]}} \to 0.
  \end{equation}
  Recalling from Proposition~\ref{prop:indepSS} that $\sT_X$ is stable with respect
  to any ample polarisation, we choose an ample Cartier divisor $H$ on $X$, a
  sufficiently large number $m$, and let $(D_j)_{1 \leq j \leq n-1} \in \vert m
  H \vert$ be general elements. Consider the corresponding general complete
  intersection curve $C := D_1 \cap \cdots \cap D_{n-1} \subsetneq X$, which
  avoids the singular locus of $X$.

  Since $K_X$ is torsion, the Kodaira-dimension of $X$ is zero,
  $\kappa(X)=0$. As $X$ has only canonical singularities, this implies that $X$
  is not covered by rational curves.  Miyaoka's Generic Semipositivity
  Theorem~\ref{miyaoka} therefore asserts that the vector bundle $\sT_X|_C \cong
  \bigl(\Omega_X^{[n-1]} \otimes \omega_X^*\bigr)\bigl|_C$ is nef. This has two
  consequences in our setup. On the one hand, since $\sE$ is a quotient of
  $\sT_X$, it follows that $\sE|_C$ and $\det \sE|_C$ are nef. On the other
  hand, since $\sE \subseteq \Omega_X^{[q-1]}$ by definition, its dual
  $\sE^*|_C$ is a quotient of $\wedge^{q-1} \sT_X|_C$, and it follows that
  $\sE^*|_C$ and $\det \sE^*|_C$ are likewise nef.  Consequently, we obtain
  $\det \sE|_C \equiv 0$. The exact Sequence~\eqref{eq:exactSomega} then implies
  that $S_\sigma$ destabilises $\sT_X$. This contradicts the assumed stability
  of $\sT_X$, and finishes the proof of Proposition~\ref{prop:max2}.
\end{proof}

\begin{cor}[Reflexive two-forms on canonical varieties with stable $\sT_X$, I]\label{cor:2formissymplectic-1}
  In the setup of Proposition~\ref{prop:max2}, $h^0 \bigl(X,\, \Omega^{[2]}_X
  \bigr) \leq 1$. 
\end{cor}
\begin{proof}
  We argue by contradiction and assume that there are two linearly independent
  forms $\sigma_1, \sigma_2 \in H^0\bigl(\wtilde X,\, \Omega^{[2]}_{\wtilde
    X}\bigr)$. Since both forms are non-degenerate by
  Proposition~\ref{prop:max2}, they induce linearly independent isomorphisms
  $\phi_\bullet : \sT_{\wtilde X} \to \Omega^{[1]}_{\wtilde X}$. The composition
  $\phi_1^{-1} \circ \phi_2$ is thus a non-trivial automorphism of $\sT_{\wtilde
    X}$. We obtain that the stable sheaf $\sT_{\wtilde X}$ is not simple,
  contradicting \cite[Cor.~1.2.8]{HL97} and thereby finishing the proof of
  Corollary~\ref{cor:2formissymplectic-1}.
\end{proof}

\begin{cor}[Reflexive two-forms on canonical varieties with stable $\sT_X$, II]\label{cor:2formissymplectic-2}
  In the setup of Proposition~\ref{prop:max2}, if there exists a non-trivial
  reflexive two-form $\sigma \in H^0 \bigl(X,\, \Omega^{[2]}_X \bigr)$, then
  $\sigma$ is a complex-symplectic form on the smooth part of $X$. In
  particular, $\dim X$ is even, $\omega_X$ is trivial, and $X$ has only rational
  Gorenstein singularities. 
\end{cor}
\begin{proof}
  Proposition~\ref{prop:max2} implies that the non-degeneracy subsheaf
  $S_{\sigma}$ vanishes. For general points $x \in X_{\reg}$, this implies that
  $\sigma|_x$ is a non-degenerate, and hence symplectic, form on the vector
  space $T_X|_x$. This already shows that the dimension of $X$ is even, say
  $\dim X = 2k$. If $\tau \in H^0 \bigl( X,\, \omega_X \bigr)$ is the section
  induced by $\wedge^k \sigma$, then $\tau$ does not vanish at $x$.

  To prove that $\sigma$ is a complex-symplectic form on the smooth part of $X$,
  we need to show that non-degeneracy holds at arbitrary points of
  $X_{\reg}$. If not, there exists a point $y \in X_{\reg}$ such that
  $\sigma|_y$ is a degenerate 2-form on the vector space $T_X|_y$. The form
  $\tau$ will therefore vanish at $y$, showing that $K_X$ can be represented by
  a non-trivial, effective $\bQ$-Cartier divisor, contradicting the assumption
  that $K_X$ is numerically trivial.

  The remaining assertions of Corollary~\ref{cor:2formissymplectic-2} follow
  immediately.
\end{proof}

\subsubsection{Exterior algebras of differential forms on the strongly stable pieces}
\label{ssec:aodf}

The algebra of differential forms on irreducible holomorphic-symplectic and
Calabi-Yau manifolds has a rather simple structure cf.~\cite[Props.~1 and
4]{Bea83}. In order to characterise the strongly stable pieces in the singular
case one would need a similar description which we formulate as the following
problem.

\begin{problem}[Forms on varieties with strongly stable tangent bundle]\label{probforms}
  Let $X$ be a normal projective variety of dimension $n > 1$ with $\omega_X
  \cong \sO_X$, having at worst canonical singularities. Assume that the tangent
  sheaf $\sT_X$ is strongly stable. Then show that the following holds.
  \begin{enumerate}
  \item\label{il:bonny} For all odd numbers $q \ne n$, we have $H^0
    \bigl(\wtilde X,\, \Omega^{[q]}_{\wtilde X} \bigr) = 0$ for all finite
    covers $f: \wtilde X \to X$, étale in codimension one.
  \item\label{il:clyde} If there exists a finite cover $g: X' \to X$, étale in
    codimension one, and an even number $0 < q < n$ such that $H^0
    \bigl(X',\,\Omega^{[q]}_{X'} \bigr) \not = 0$, then there exists a reflexive
    $2$-form $\sigma' \in H^0\bigl(X', \Omega_{X'}^{[2]} \bigr)$, symplectic on
    the smooth locus $X'_{\reg}$, such that for any finite cover $f: \wtilde X
    \to X'$, étale in codimension one, the exterior algebra of global reflexive
    forms on $\wtilde X$ is generated by $f^*(\sigma')$. In other words,
    $$
    \bigoplus_p H^0\Bigl(\wtilde X, \Omega_{\wtilde X}^{[p]} \Bigr) = \bC \bigl[
    f^*(\sigma) \bigr].
    $$
  \end{enumerate}
\end{problem}

\begin{rem}\label{rem:stableqvanishes}
  Notice that the assumptions on the strong stability of $\sT_X$ and on the
  dimension of $X$ automatically imply $\wtilde q(X) = 0$.
\end{rem}

As we will will discuss in more detail in the subsequent
Section~\ref{sect:singularclassification}, a positive solution to
Problem~\ref{probforms} leads to a characterisation of canonical varieties with
trivial canonical class and strongly stable tangent bundle as singular analogues
of Calabi-Yau or irreducible holomorphic symplectic manifolds. There are a
number of cases where Problem~\ref{probforms} can be solved. We conclude
Section~\ref{ssec:aodf} with two propositions that provide evidence by
discussing the case where $X$ is smooth, or of dimension $\leq 5$, respectively.

\begin{prop}\label{prop:smoothProblem}
  The claims of Problem~\ref{probforms} hold if $X$ is smooth.
\end{prop}
\begin{proof}
  Note that on a smooth variety $X$ the sheaves $\Omega_X^p$ and
  $\Omega_X^{[p]}$ coincide and that any finite cover of $X$ that is étale in
  codimension one is actually étale by purity of the branch locus.

  Let now $X$ be a smooth projective variety of dimension $n$ with $\omega_X
  \cong \sO_X$.  Assume that the tangent bundle $\sT_X$ is strongly stable. As
  noticed in Remark~\ref{rem:stableqvanishes}, this implies that $\wtilde q (X)
  = 0$. Consequently, the fundamental group of $X$ is finite by
  \cite[Thm.~2(2)]{Bea83}. Let $\what X \to X$ be the universal cover. Since
  $\sT_X$ is strongly stable, $\sT_{\what X}$ is stable with respect to any
  polarisation, and $\what X$ is hence irreducible in the sense of the
  Beauville-Bogomolov decomposition Theorem~\ref{bb}. As $\wtilde q (X) = 0$,
  the manifold $\what X$ is therefore either Calabi-Yau or irreducible
  holomorphic-symplectic.

  In order to show (\ref{probforms}.\ref{il:bonny}), pulling back forms from any
  étale cover $\wtilde X \to X$ to the universal cover $\what X$ if necessary,
  it suffices to note that both in the Calabi--Yau and in the irreducible
  holomorphic-symplectic case, $\what X$ does not support differential forms of
  odd degree $p<n$ by \cite[Props.~1 and 3]{Bea83}.

  To show (\ref{probforms}.\ref{il:clyde}), let $X' \to X$ be any étale cover,
  and assume that there exists a non-vanishing form such that that $\sigma' \in
  H^0 \bigl( X',\,\Omega^{[q]}_{X'} \bigr)$ for some even number $0 < q <
  n$. Pulling back $\sigma'$ to the universal cover $\what X$, we see that
  $\what X$ cannot be Calabi--Yau and is therefore irreducible
  holomorphic-symplectic, say with symplectic form $\what \sigma$. Consequently,
  \cite[Prop.~3]{Bea83} implies that the algebra of differential forms on $\what
  X$ is generated by $\what \sigma$. Hence, in order to establish the claim it
  therefore suffices to show that $\what X$ is biholomorphic to $X$ and
  therefore also to $X'$. In other words, we need to show that $X$ is already
  simply-connected. This is done in Lemma~\ref{lem:alreadysymplectic} below.
\end{proof}

We are grateful to Keiji Oguiso for pointing us towards \cite{OguisoSchroeer}
and for explaining the following observation to us.

\begin{lem}\label{lem:alreadysymplectic}
  Let $X$ be a projective manifold whose universal cover is an irreducible
  holomorphic-symplectic manifold. If the canonical bundle of $X$ is trivial,
  $\omega_X \cong \sO_X$, then $X$ is simply-connected, and therefore itself
  irreducible holomorphic-symplectic.
\end{lem}
\begin{proof}
  The assumptions on $X$ imply that $X$ is an \emph{Enriques manifold} in the
  sense of Oguiso and Schröer~\cite{OguisoSchroeer}, see also
  \cite{BoissiereNieperSarti}. Since the canonical bundle of $X$ is trivial, and
  since the universal cover of $X$ is irreducible holomorphic-symplectic, the
  fundamental group of $X$ is finite, cf.~\cite[Thm.~2(2)]{Bea83}. Let $d$
  denote the degree of the universal covering map $\what X \to X$, and set $\dim
  X = \dim \what X = n = 2k$. It then follows from
  \cite[Prop.~2.4]{OguisoSchroeer} that $d\, \vert\, (k+1)$. Moreover, since $X$
  is assumed to have trivial canonical bundle, \cite[Prop.~2.6]{OguisoSchroeer}
  implies that additionally $d\, \vert\, k$. Consequently, we have $d =1$, which
  proves the claim.
\end{proof}

\begin{prop}
  The claims of Problem~\ref{probforms} hold if $\dim X \leq 5$.
\end{prop}
\begin{proof}
  Let $X$ be a projective variety of dimension greater than one, having at worst
  canonical singularities. Assume that $X$ has a trivial canonical bundle,
  $\omega_X \cong \sO_X$, and a strongly stable tangent sheaf $\sT_X$.  Again we
  have $\wtilde q(X) = 0$, since $\sT_X$ is strongly stable. Fix a finite cover
  $\wtilde X \to X$, étale in codimension one.

  If $\dim X = 3$, then Corollary~\ref{cor:max1} immediately implies that $h^0
  \bigl(\wtilde X, \, \Omega_{\wtilde X}^{[1]} \bigr) = h^0 \bigl( \wtilde X, \,
  \Omega_{\wtilde X}^{[2]} \bigr) =
  0$. Conditions~(\ref{probforms}.\ref{il:bonny}) and
  (\ref{probforms}.\ref{il:clyde}) of Problem~\ref{probforms} are therefore
  satisfied.

  Now assume that $\dim X = 4$. In this setting, Corollary~\ref{cor:max1} gives
  that $h^0 \bigl(\wtilde X, \, \Omega_{\wtilde X}^{[1]} \bigr) = h^0 \bigl(
  \wtilde X, \, \Omega_{\wtilde X}^{[3]} \bigr) = 0$. The claims of
  Problem~\ref{probforms} thus follow from
  Corollary~\ref{cor:2formissymplectic-2} and from the fact that
  $h^0\bigl(\wtilde X,\, \Omega^{[2]}_{\wtilde X}\bigr) \leq 1$, as shown in
  Corollary~\ref{cor:2formissymplectic-1}.

  It remains to consider the case where $\dim X = 5$, where
  Corollary~\ref{cor:max1} asserts that $h^0 \bigl(\wtilde X, \, \Omega_{\wtilde
    X}^{[1]} \bigr) = h^0 \bigl( \wtilde X, \, \Omega_{\wtilde X}^{[4]} \bigr) =
  0$. The claims of Problem~\ref{probforms} will follow once we show that $h^0
  \bigl(\wtilde X, \, \Omega_{\wtilde X}^{[2]} \bigr)$ and $h^0 \bigl( \wtilde
  X, \, \Omega_{\wtilde X}^{[3]} \bigr)$ vanish as well. For that, recall from
  item (\ref{cor:64}.\ref{il:trout}) of Corollary~\ref{cor:64} that there exists
  a non-trivial $3$-form on $\widetilde X$ if and only if there exists a
  non-trivial reflexive $2$-form on $\widetilde X$. However, by
  Corollary~\ref{cor:2formissymplectic-2} any non-trivial $2$-form would be
  non-degenerate, forcing $\dim X$ to be even, a contradiction.
\end{proof}

\subsubsection{Calabi-Yau and holomorphic-symplectic varieties}
\label{sect:singularclassification}

The following definition is motivated by the description of the exterior algebra
of Calabi-Yau manifolds and irreducible holomorphic-symplectic manifolds,
\cite[Props.~1 and 4]{Bea83}, and by the discussion of Problem~\ref{probforms}
in the previous section.

\begin{defn}[Calabi-Yau and symplectic varieties in the singular case]\label{def:CYSympl}
  Let $X$ be a normal projective variety with $\omega_X \cong \sO_X$, having at
  worst canonical singularities.
  \begin{enumerate}
  \item\label{il:CY} We call $X$ \emph{Calabi-Yau} if $H^0 \bigl(\wtilde X, \,
    \Omega_{X}^{[q]} \bigr) = 0$ for all numbers $0 < q < \dim X$ and all finite
    covers $\wtilde X \to X$, étale in codimension one.
  \item\label{il:symplectic} We call $X$ \emph{irreducible
      holomorphic-symplectic} if there exists a reflexive $2$-form $\sigma \in
    H^0\bigl(X, \Omega_{X}^{[2]} \bigr)$ such that $\sigma$ is everywhere
    non-degenerate on $X_{\reg}$, and such that for all finite covers $f:
    \wtilde X \to X$, étale in codimension one, the exterior algebra of global
    reflexive forms is generated by $f^*(\sigma)$.
  \end{enumerate}
\end{defn}

\begin{rem}[Augmented irregularity of Calabi-Yau and symplectic varieties]
  If $X$ is Calabi-Yau or irreducible holomorphic-symplectic in the sense of
  Definition~\ref{def:CYSympl}, it follows immediately that the augmented
  irregularity of $X$ vanishes, $\wtilde q(X) = 0$.
\end{rem}

\begin{rem}[Definition~(\ref{def:CYSympl}.\ref{il:CY}) for ``Calabi-Yau'' in the smooth case]
  By \cite[Sect.~3, Prop.~2]{Bea83} the conditions spelled out in
  (\ref{def:CYSympl}.\ref{il:CY}) are in the smooth case equivalent to the
  existence of a Kähler metric with holonomy $SU(m)$. If $X$ is smooth and
  Calabi-Yau in the sense of Definition~\ref{def:CYSympl}, then $X$ is not
  necessarily simply-connected, but may have finite fundamental group. If we
  assume additionally that $\dim X$ is even, then a simple computation with
  holomorphic Euler characteristics shows that $X$ is in fact simply-connected,
  cf.~\cite[Prop.~2 and Rem.]{Bea83}.
\end{rem}

\begin{rem}[Definition~(\ref{def:CYSympl}.\ref{il:symplectic}) for ``irreducible symplectic'' in the smooth case]
  If $X$ is smooth and irreducible holomorphic-symplectic in the sense of
  Definition~\ref{def:CYSympl}, then $X$ is simply-connected. In fact, even
  without the condition on the algebra of differential forms on étale covers, if
  $X$ is a holomorphic-symplectic manifold of complex dimension $2n$ such that
  $$
  H^{k,0}(X) \cong \left\{ 
    \begin{aligned}
      \bC & \,\,\,\, \text{ if $k$ is even} \\
      0 &  \,\,\,\, \text{ if $k$ is odd,}
    \end{aligned}
  \right.
  $$
  then $X$ is simply-connected, that is, $X$ is an irreducible
  holomorphic-symplectic manifold, see \cite[Prop.~A.1]{HuybrechtsNieper2011}.
\end{rem}

Assuming Problem~\ref{probforms} can be solved, the following two propositions
provide a classification of the strongly stable pieces in the conjectural
version of the Beauville-Bogomolov decomposition for the singular case.

\begin{prop}[Characterisation of strongly stable pieces, I]\label{prop:chsp1}
  Let $X$ be Calabi-Yau or irreducible holomorphic-symplectic in the sense of
  Definition~\ref{def:CYSympl}. Then $\sT_X$ is strongly stable in the sense of
  Definition~\ref{def:strongStab}.
\end{prop}
\begin{proof}
  Let $X$ be Calabi-Yau or irreducible symplectic. We argue by contradiction and
  assume that there exists a finite cover $g: \wtilde X \to X$, étale in
  codimension one, and ample Cartier divisors $\wtilde H_1, \ldots \wtilde
  H_{n-1}$ on $\wtilde X$ such that the tangent sheaf $\sT_{\wtilde X}$ is not
  stable with respect to the $\wtilde H_i$. In this setting, Theorem~\ref{deco}
  asserts that there exists a further finite cover $h: \what X \to \wtilde X$
  and a proper decomposition
  \begin{equation}\label{eq:dCC}
    \sT_{\what X} \cong  \sE \oplus \sF.    
  \end{equation}
  with $\det \sE \cong \sO_{\what X}$. Setting $r := \rank \sE$, the
  splitting~\eqref{eq:dCC} immediately gives an embedding $\sO_{\what X} \cong
  \det \sE^* \into \Omega^{[r]}_{\what X}$, and an associated form $\tau \in H^0
  \bigl( X,\, \Omega^{[r]}_{\what X} \bigr)$. Since $0 < r < \dim X$, it follows
  that $X$ cannot be Calabi-Yau.

  Since $\sF$ is contained in the degeneracy subsheaf $S_\tau$, as introduced in
  Definition~\ref{cont}, it is clear that $\tau$ cannot be a wedge-power of the
  pull-back of any symplectic form on $X$. This rules out that $X$ is
  irreducible holomorphic-symplectic in the sense of
  Definition~\ref{def:CYSympl}. We obtain a contradiction, which finishes the
  proof of Proposition~\ref{prop:chsp1}.
\end{proof}

A positive solution to Problem~\ref{probforms} would immediately give a partial
converse to Proposition~\ref{prop:chsp1}.

\begin{prop}[Characterisation of strongly stable pieces, II]
  Let $X$ be a normal projective variety with $\omega_X = \sO_X$ having at worst
  canonical singularities. Assume that $\sT_X$ is strongly stable. If the
  assertions of Problem~\ref{probforms} hold, then either
  \begin{enumerate}
  \item the semistable sheaf $\wedge^{[2]} \sT_X$ is strongly stable, and $X$ is
    Calabi-Yau, or
  \item there exists a finite cover $\wtilde X \to X$, étale in codimension one,
    such that the sheaf $\wedge^{[2]} \sT_{\wtilde X}$ is not $\wtilde H$-stable
    for some polarisation $\wtilde H$ on $\wtilde X$, and $\wtilde X$ is
    irreducible holomorphic-symplectic. \qed
  \end{enumerate}
\end{prop}

\begin{rem}\label{rem:smoothclassification}
  In the second case of the previous proposition one would of course rather like
  $X$ itself to be irreducible holomorphic-symplectic.  This is in fact true if
  $X$ is additionally assumed to be smooth: If $X$ is not Calabi--Yau, then the
  universal cover $\widetilde X$ of $X$ is irreducible
  holomorphic-symplectic. Since additionally the canonical bundle $\omega_X$ is
  assumed to be trivial, Lemma~\ref{lem:alreadysymplectic} implies that $X$
  itself is irreducible holomorphic-symplectic.
\end{rem}

\subsection{Fundamental groups of varieties with trivial canonical class}
\label{subsect:fundamentalgroups}

A Kähler manifold $X$ with trivial canonical class and vanishing augmented
irregularity $\wtilde q(X)$ has finite fundamental group, see
\cite[Thm.~1]{Bea83}. We believe that the same should hold for projective
varieties, in our singular setting. We show that this is true, at least under
the assumption that $\chi(X,\sO_X) \ne 0$.

\begin{prop}[Fundamental groups of canonical varieties with $K_X \equiv 0$, I]\label{prop:finitefundamental}
  Let $X$ be a normal projective variety with at worst canonical
  singularities. If $K_X$ is torsion and if $\chi(X,\sO_X) \ne 0$, then
  $\pi_1(X)$ is finite, of cardinality 
  $$
  |\pi_1(X)| \leq \frac{2^{n-1}}{|\chi(X, \sO_X)|}.
  $$
\end{prop}

\begin{rem}
  If $X$ is smooth and $K_X$ is torsion, then then classical Beauville-Bogomolov
  Decomposition Theorem~\ref{bb} together with
  Proposition~\ref{prop:finitefundamental} shows that $\chi(X,\sO_X) \ne 0$
  implies $\wtilde q(X) = 0$.
\end{rem}

\begin{proof}[Proof of Proposition~\ref{prop:finitefundamental}]
  Set $n := \dim X$.  Let $\pi: \wtilde X \to X$ be a strong resolution of
  singularities. Recalling from \cite[Thm.~5.22]{KM98} that $X$ has rational
  singularities, we obtain that $\chi(\wtilde X,\sO_{\wtilde X}) = \chi(X,\sO_X)
  \ne 0$. Consider the invariant
  $$
  \kappa^+ \bigl(\wtilde X \bigr) := \max \left\{ \kappa \bigl( \det \sF \bigr)
    \, \bigl| \, \sF \text{ is a coherent subsheaf of $\Omega^p_{\wtilde X}$,
      for some $p$} \right\}.
  $$
  We are going to show that $\kappa^+ \bigl(\wtilde X \bigr) = 0$.  Using the
  assumption that $\chi(\wtilde X,\sO_{\wtilde X}) \ne 0$, Campana has then
  shown in \cite[Cor.~5.3]{Ca95} that $\pi_1(\wtilde X)$ is finite, of
  cardinality at most $2^{n-1} \cdot |\chi(\wtilde X, \sO_{\wtilde
    X})|^{-1}$. Since the natural map $\pi_1\bigl(\wtilde X \bigr) \to \pi_1(X)$
  is isomorphic by \cite[Thm.~1.1]{Takayama2003}, this implies that $\pi_1(X)$
  is likewise finite of the same cardinality.

  So let $0 \leq p \leq n$ be any number and let $\sF \subseteq
  \Omega^p_{\wtilde X}$ be a coherent subsheaf. As a subsheaf of a torsion-free
  sheaf, $\sF$ is itself torsion-free, and therefore locally free in codimension
  one. Next, let $C \subset X$ be a general complete intersection curve. Recall
  that the strong resolution map $\pi$ is isomorphic along $C$, and denote the
  preimage curve by $\wtilde C := \pi^{-1}(C)$. The restricted sheaves
  $\Omega^p_{\wtilde X}\bigl|_{\wtilde C}$ and $\sF|_{\wtilde C}$ are then both
  locally free.

  Since $K_X$ is torsion, the Kodaira-dimension of $X$ is zero,
  $\kappa(X)=0$. As $X$ has only canonical singularities, this shows that $X$ is
  not covered by rational curves. Miyaoka's Generic Semipositivity
  Theorem~\ref{miyaoka} therefore implies that $\Omega^q_{\wtilde X}|_{\wtilde
    C}$ is nef for all $q$. Better still, we have $\deg \Omega^n_{\wtilde
    X}|_{\wtilde C} = 0$, so that
  $$
  \bigl( \Omega^p_{\wtilde X}|_{\wtilde C} \bigr)^* \cong \wedge^p
  \sT_X|_{\wtilde C} \cong \Hom \Bigl(\Omega^n_{\wtilde X}|_{\wtilde C},\,
  \Omega^{n-p}_{\wtilde X}|_{\wtilde C} \Bigr) \cong \bigl( \Omega^n_{\wtilde
    X}|_{\wtilde C} \bigr)^* \otimes \Omega^{n-p}_{\wtilde X}|_{\wtilde C}
  $$
  is likewise as a nef vector bundle on the curve $\wtilde C$. Its quotient
  $\sF^*|_{\wtilde C}$ is then nef as well.  In summary, we obtain that $
  c_1(\sF) \cdot \wtilde C \leq 0 $.  Since the curves $\wtilde C$ are moving,
  this implies $\kappa (\det \sF) \leq 0$, and therefore $\kappa^+ \bigl(\wtilde
  X \bigr) \leq 0$. Since $\kappa \bigl(\wtilde X\bigr) = 0$, we obtain
  $\kappa^+ \bigl(\wtilde X\bigr) = 0$, as claimed. This finishes the proof of
  Proposition~\ref{prop:finitefundamental}.
\end{proof}

\begin{cor}[Fundamental groups of canonical varieties with $K_X \equiv 0$, II]\label{cor:almostLast}
  Let $X$ be a normal projective variety with at worst canonical singularities.
  Assume that $\dim X \leq 4$, and that the canonical divisor $K_X$ is
  numerically trivial. Then $\pi_1(X)$ is almost Abelian, that is, $\pi_1(X)$
  contains an Abelian subgroup of finite index.
\end{cor} 
\begin{proof}
  Recall from \cite[4.17.3]{Kollar95s} that the statement of
  Corollary~\ref{cor:almostLast} is well-known if $\dim X \leq 3$. We will
  therefore assume for the remainder of the proof that $X$ is of dimension four.

  Let $f: \wtilde X \to X$ be the index-one cover associated with $K_X$. As we
  have noted before, $f$ is étale in codimension one, $\wtilde X$ has canonical
  singularities, and $\omega_{\wtilde X} = \sO_{\wtilde X}$, cf.~\cite[5.19 and
  5.20]{KM98}. The image of the natural map $\pi_1(\wtilde X) \to \pi_1(X)$ has
  finite index in $\pi_1(X)$, cf.~\cite[Prop.~2.10(2)]{Kollar95s} and
  \cite[Prop.~1.3]{Campana91}. Replacing $X$ by $\wtilde X$, if necessary, we
  may therefore assume without loss of generality that $\omega_X$ is
  trivial. Passing to a further cover, Corollary~\ref{cor:qgleichnull} even
  allows to assume that $X$ is of the form $X = A \times Z$, where $A$ is an
  Abelian variety, and $Y$ is normal projective variety with at worst canonical
  singularities, with trivial canonical class and vanishing augmented
  irregularity, $\omega_Z \cong \sO_Z$ and $\wtilde q(Z) =0$.

  If $\dim Z \leq 3$, then \cite[4.17.3]{Kollar95s} asserts that $\pi_1(Z)$ is
  almost Abelian. Since $\pi_1(A)$ is Abelian, this finishes the proof.

  It remains to consider that case where $\dim Z = 4$, that is, where $X=Z$ and
  $\wtilde q(Z)=0$. In this case, we finish proof by showing that the
  fundamental group of $X$ is finite.  Recall from Corollary~\ref{cor:max1} that
  $X$ does not carry any reflexive $1$-form or $3$-forms. Using
  Proposition~\ref{prop:forms-2} to relate $H^p\bigl(X,\, \sO_X\bigr)$ with the
  space of reflexive $p$-forms we see that $\chi(X,\mathcal O_X) > 0$, and we
  conclude by Proposition~\ref{prop:finitefundamental} that $\pi_1(X)$ is
  finite, thus finishing the proof of Corollary~\ref{cor:almostLast}.
\end{proof} 

\begin{rem}[Fundamental groups of smooth 4-folds with $\kappa=0$]
  If $X$ is a smooth projective $4$-fold with $\kappa (X) = 0$ admitting a good
  minimal model $X'$, then $\pi_1(X) = \pi_1(X')$ by
\cite[Thm.~1.1]{Takayama2003} or  \cite[Thm.~7.8.1]{Kollar93}. Consequently, $\pi_1(X)$ is almost Abelian.
\end{rem}

Assuming that the claims of Problem~\ref{probforms} hold, the following
corollary complements the results obtained in Section~\ref{sec:kawamata}, and in
particular the results of Corollary~\ref{cor:64}.

\begin{cor}[Fundamental groups of even-dim.\ $X$ with $\sT_X$ strongly stable]
  Let $X$ be a normal projective variety with $\omega_X \cong \sO_X$ having at
  worst canonical singularities. Suppose furthermore that $\dim X$ is even and
  that $\sT_X$ is strongly stable. If Problem~\ref{probforms} has a positive
  solution, then $\chi(X,\sO_X) > 0$ and $\pi_1(X)$ is finite. \qed
\end{cor}

\providecommand{\bysame}{\leavevmode\hbox to3em{\hrulefill}\thinspace}
\providecommand{\MR}{\relax\ifhmode\unskip\space\fi MR }
\providecommand{\MRhref}[2]{%
  \href{http://www.ams.org/mathscinet-getitem?mr=#1}{#2}
}
\providecommand{\href}[2]{#2}

\end{document}